\documentclass[a4paper]{amsart}
\usepackage[margin=3.5cm]{geometry}
\usepackage[utf8]{inputenc}
\usepackage{amsmath,amssymb,amsthm,amscd,amsfonts,mathrsfs,verbatim, mathtools}
\usepackage{stmaryrd}
\usepackage{enumitem}
\usepackage[all]{xy}
\usepackage{tikz}
\usetikzlibrary{decorations.pathreplacing}

\newcommand{\Q}{{\mathbb Q}}                   
\newcommand{\C}{{\mathbb C}}                   
\newcommand{\ZZ}{{\mathbb Z}}                   
\newcommand{\QQ}{{\mathbb Q}}                   
\newcommand{\CC}{{\mathbb C}}                   

\newcommand{\CP}[1]{\mathbb{P}^{#1}}      

\DeclareMathOperator{\Spec}{Spec}

\newtheorem{theorem}{Theorem}[section]

\newtheorem{proposition}[theorem]{Proposition}
\newtheorem{corollary}[theorem]{Corollary}
\newtheorem{lemma}[theorem]{Lemma}
\newtheorem{conjecture}[]{Conjecture}
\newtheorem{defn}[theorem]{Definition}

\theoremstyle{definition}
\newtheorem{definition}[theorem]{Definition}

\theoremstyle{remark}
\newtheorem{remark}[theorem]{Remark}

\newtheorem{example}[theorem]{Example}

\title{Blow-ups and the quantum spectrum of surfaces}

\makeatletter
\@namedef{subjclassname@2020}{%
  \textup{2020} Mathematics Subject Classification}
\makeatother

\author{\'Ad\'am Gyenge}
\address{Department of Algebra and Geometry, Institute of Mathematics, 
Budapest University of Technology and Economics, 
M\H{u}egyetem rakpart 3, 1111, 
Budapest, Hungary}
\email{Gyenge.Adam@ttk.bme.hu}

\author{Szil\'ard Szab\'o}
\address{Institute of Mathematics, E\"otv\"os Lor\'and University, P\'azm\'any P\'eter s\'et\'any 1/C, 1117, Budapest, Hungary and Alfr\'ed R\'enyi Institute of Mathematics, Re\'altanoda utca 13-15, 1053, Budapest, Hungary}
\email{szilard.szabo@ttk.elte.hu}


\begin{document}

\begin{abstract}
We investigate the behaviour of the spectrum of the quantum (or Dubrovin) connection of smooth projective surfaces under blow-ups. 
Our main result is that for small values of the parameters, the quantum 
spectrum of such a surface is asymptotically the union of the quantum 
spectrum of a minimal model of the surface and a finite number of additional
points located ``close to infinity'', that correspond bijectively to the 
exceptional divisors. 
This proves a conjecture of Kontsevich in the surface case.
\end{abstract}

\maketitle

\tableofcontents 

\section{Introduction}

Let $X$ be a nonsingular projective variety over $\C$. 
Genus $g$ Gromov--Witten invariants of $X$ count holomorphic maps from 
curves of genus $g$ into $X$, along with certain degenerations. 
Since they were first 
introduced~\cite{witten1990two,kontsevich1994gromov, ruan1995mathematical},
these invariants have attracted great interest, and the literature of the 
area has become extensive. 

A particularly elegant geometric formulation of the information contained 
by genus $0$ Gromov-Witten invariants of $X$ was discovered by 
Dubrovin~\cite{dubrovin1996geometry}, who used them to construct a 
Frobenius manifold structure on $H^{\ast}(X,\C)$ (under some conditions). 
Among other data, this involves a flat metric, with Levi--Civita connection 
denoted by $\nabla$. 
The general theory of Frobenius manifolds then shows that $\nabla$ 
admits a meromorphic flat deformation parameterized by a parameter $u\in\C$, 
that is, a connection on the vector bundle $\mathcal{H}$ with fiber 
$H^{\ast}(X,\C)$ over $\C \times H^{\ast}(X,\C)$. 
This flat deformation is called the quantum (or Dubrovin) connection of $X$; 
for a more precise description see Section~\ref{subsec:dubrovin} below. 
For ease of notation, we denote the Dubrovin connection by $\nabla$ too. 
For a cohomology class $\tau \in H^{\ast}(X,\C)$, the deformed flat 
connection on the restriction $\mathcal{H}_{\tau}$ of $\mathcal{H}$ to 
$\C\times\{ \tau \}$ is of the form 
\[\nabla_{\frac{\partial}{\partial u}}^{(\tau)}=\frac{\partial}{\partial u} + \frac{1}{u^2}K *_{\tau}+\frac{1}{u}G,\]
where
$K$ is a deformation of the operator which takes quantum product
(see Definition~\ref{def:quantum_product}) 
with the first Chern class of $X$, and $G$ is a certain grading operator. For general $X$, the deformation $\nabla$ is not known to be convergent for 
general $\tau\neq 0$, but in certain toric cases this is known,
see~\cite{coates2015convergence}. 

In this article we consider the case of smooth projective surfaces $X$.
If $X$ is a rational surface, then it is proved in~\cite[Section~2.4,~Remark~3]{itenberg2004logarithmic, itenberg2005logarithmic} using geometric ideas that convergence holds for a special choice of one of 
the parameters (namely, when $t_p = 1$ in our notations). 
In the Appendix, we give an independent algebraic proof of convergence of the 
potential for surfaces in a neighbourhood of $t_p = 0$ solely based on the 
known recurrence relations for the GW-invariants. 

For fixed $\tau$, the Dubrovin connection 
$\nabla^{(\tau)}$ has a logarithmic singularity 
at $u = \infty$, but has an order two pole, and hence an irregular singularity 
of Poincar\'e--Katz rank $\leq 1$ at $u=0$. 
Write $QC(X)_{\tau} \coloneqq (\mathcal{H}_{\tau},\nabla^{(\tau)})$.
By virtue of the Turrittin--Hukuhara--Levelt (THL)  theorem~\cite{turrittin1955convergent,hukuhara1942points,levelt1975jordan}, 
in the unramified case 
the restriction of $QC(X)_{\tau}$ to the formal neighbourhood of $u=0$ 
admits the following orthogonal decomposition 
\begin{equation}\label{eq:THLdecomp}
 QC(X)_{\tau}\otimes_{\mathbb{C}\{ u\}} \mathbb{C}\llbracket u \rrbracket \cong 
\oplus_{\lambda \in \mathrm{Spec}(K)} 
(\mathbb{C}\llbracket u \rrbracket ,d+d(\lambda/u)) 
\otimes_{\mathbb{C}\{ u \}} \mathcal{F}_{\lambda} 
\end{equation}
where $\mathrm{Spec}(K) = \mathrm{Spec}(K_{X,\tau})$ is the set of eigenvalues of $K*_{\tau}$ on $H^{\ast}(X,\C)$ and $\mathcal{F}_{\lambda}$ is a free $\mathbb{C}\{u\}$-module with regular singular connection. 
(Of course, the spectrum and the corresponding regular singular factors depend 
on $\tau$, however we omit to spell out this dependence for ease of notation.) 
We note that if $K$ is regular semi-simple then $QC(X)_{\tau}$ is 
unramified and in particular the above decomposition holds with rank $1$ 
connections $\mathcal{F}_{\lambda}$.
This observation motivates the study of the spectrum of the operator $K$. 
The question is closely related to the \emph{Gamma conjecture} of Dubrovin and to semiorthogonal decompositions of the K-group $K(X)$. For a concise overview of the picture, see for example~\cite[Section 8]{iritani2020global} or~\cite[Section~1.1]{Cotti_Dubrovin_Guzzetti}. 
The spectrum of $K$ is often called the \emph{quantum spectrum} of $X$.

Let $Z \subset X$ be a smooth closed subvariety of codimension $n\geq 2$, and let  $\widetilde{X}$ be the blow-up of $X$ in $Z$.
It is known from the work of Orlov~\cite{orlov1993projective} that the K-group of 
$\widetilde{X}$ can be expressed from those of $X$ and $Z$. 
More precisely, the bounded derived categories of coherent sheaves have 
a semiorthogonal decomposition
\[ D(\widetilde{X})=\langle D(X),\underbrace{D(Z),\dots,D(Z)}_{(n-1)\textrm{ times}} \rangle.\]

It is a natural question to ask whether this decomposition is 
reflected at the level of the quantum connection. 
A conjecture of M.~Kontsevich in this 
direction can roughly be stated as follows.
\begin{conjecture}[{Blow-up conjecture for the quantum spectrum, M.~Kontsevich,~\cite[pp.~13-14]{kontsevich2021blowup}}]
\label{conj:BU}
Let $A \subset \mathbb{C}$ be a fixed compact set containing $\mathrm{Spec}(K_{X,\tau})$. 
There exists a nonlinear formal invertible map 
\begin{align*}
    H^{\ast}(\widetilde{X},\C) & \to H^{\ast}(X,\C) \oplus H^{\ast}(Z,\C)^{\oplus (n-1)} \\
    \tilde{\tau} & \mapsto (\tau, \tau') 
\end{align*}
such that the following conditions hold. 
\begin{enumerate}
    \item The intersection of $\mathrm{Spec}(K_{\widetilde{X}, \tilde{\tau}})$ 
    with $A$ converges to $\mathrm{Spec}(K_{X,\tau})$, as $\tilde{\tau}\to 0$. 
    \item For $|\tilde{\tau}|\ll 1$, the set $\mathrm{Spec}(K_{\widetilde{X}, \tilde{\tau}})\setminus A$ is equal to the union of a copy of $\mathrm{Spec}(K_{Z,\tau'})$ for some $\tau' = \tau'(\tilde{\tau})$ and its multiples by all $(n-1)$-th roots of unity, where $n = \operatorname{codim}_X(Z)$. 
\end{enumerate}
\end{conjecture}

Our main result confirms (a slightly refined version of) this expectation for surfaces $X$. 
For its precise statement, let $X_{\min}$ be a minimal model of the surface $X$, 
so that $X_r= X$ can be obtained from $X_{\min}$ by blowing up a finite number $r$ of points. 
Our main result compares the spectrum of $K_r$ constructed from $QC( X_r)$ with the spectrum of $K_{\min}$ constructed from $QC(X_{\min})$.  
As we will see, $QC(X_{\min})$ is defined over 
\[\Spec \CC \{ q_1,\dots,q_m,t_p \} \] 
while $QC(X_r)$ is defined over its extension 
\[\Spec \CC \{ q_1,\dots,q_m,q_{m+1},\dots,q_{m+r},t_p \} [q_{m+1}^{-1},\dots,q_{m+r}^{-1}] \]
where $m$ is the second Betti number of $X_{\min}$, and where we denote by $\C \{ q, t \}$ the ring of power series with a positive radius of convergence. 
Let us denote by  
$\{\lambda_0,\dots,\lambda_{m+1}\}$ and $\{\mu_0,\dots,\mu_{m+r+1}\}$ 
the spectra of $K_{\min}$ and $K$ respectively.

\begin{theorem}[Proposition~\ref{prop:P2domain}, {Theorem~\ref{thm:main}}]  
\label{thm:intromain}
Let $X_r$ be a smooth projective surface with a minimal model $X_{\min}$.
Let 
\[U \subset \Spec \CC\{ q_1,\dots ,q_{m+r} \}\] 
be any closed conical subset for the analytic topology that 
intersects the coordinate hyperplanes $\{q_i = 0\}$, $1 \leq i \leq m+r$ only at the origin. 
We assume $(q_1,\dots ,q_{m+r})\in U$ and set $\tau = (t_0, q_1,\dots ,q_{m+r}, t_p)$. 
\begin{enumerate}
\item There exists $\varepsilon>0$ such that $\nabla^{(\tau)}$ converges for $\Vert (q_1,\dots ,q_{m+r}) \Vert < \varepsilon, |t_p |<1$. 
\item For fixed $|t_p|<1$, the spectrum of $K_r$ converges to the 
union of the spectrum of $K_{\min}$ and the inverses of the new variables 
$q_{m+1},\dots,q_{m+r}$, as $(q_1, \ldots , q_{m+r})\to \vec{0}$ in $U$. 
Namely, up to a suitable relabeling,
$$
    \lim \frac{\lambda_j (q_1,  \ldots , q_{m}, t_p )}
    {\mu_j (q_1,  \ldots , q_{m+r}, t_p)} = 1 
    \quad (0 \leq j \leq m+1 )
$$
and 
$$
    \lim \mu_j (q_1,  \ldots , q_{m+r}, t_p ) q_{j-1} = 1 \quad (m+2 \leq j \leq m+r+1). 
$$
\end{enumerate}
\end{theorem}


In \cite{iritani2023quantum} a complete proof of Conjecture~\ref{conj:BU} is proposed. 
The approach of H.~Iritani is based on Fourier analysis of equivariant quantum cohomology in the spirit of~\cite{TelemanICMtalk}, applied to the $\mathbb{C}^{\times}$-variety $\operatorname{Bl}_{Z\times \{ 0 \}} (X \times \CP1)$. 
Our Theorem~\ref{thm:main} gives an alternative, more elementary and more explicit proof in the surface case by combinatorially analyzing the Newton polygon of the characteristic polynomial of the quantum connection matrix. 

Our main Theorem~\ref{thm:main} also gives a new, alternative proof of \cite{bayer2004semisimple} in the surface case. Our calculations are however more general and completely explicit in terms of the spectral parameters. We also investigate the quantum spectrum of the minimal model to get a complete picture on the quantum spectrum of any smooth surface.

It is worth to point out that the last equality of Theorem~\ref{thm:intromain} says that the new eigenvalues of $K_r$ (as compared to those of $K_{\min}$) converge to the inverse of the corresponding $q$-variable. In particular, it is essential to take the limit inside the region $U$. The existence of negative exponents has already been observed in the literature, see for instance \cite[Section~2]{gottsche1998quantum} or \cite[Section~4]{bayer2004semisimple}. On the other hand, the degree $0$ terms in the expansion of the Gromov--Witten potential, and hence in that of $K_r$, come from cup product in cohomology.

The structure of the paper is as follows. In Section~\ref{sec:notback} we recall the relevant facts about Gromov--Witten (sometimes abbreviated as GW) 
invariants and the Dubrovin connection as well as the main results of \cite{gottsche1998quantum, hu2006quantum} on the behaviour of GW invariants of surfaces under blow-ups. In Section~\ref{sec:opK} we analyze the structure of the operator $K$. In Section~\ref{sec:newtonpoly}  we infer about the behaviour of the Newton polygon of the characteristic polynomial of $K$ using a case-by-case analysis according to the three main classes of minimal surfaces (the projective plane, ruled surfaces and surfaces with numerically effective canonical class). In Section~\ref{sec:spectrum}
we prove our main Theorem~\ref{thm:main}. 
The Appendix is devoted to a new proof of analyticity of the GW potential for rational surfaces. 

\subsection*{Acknowledgements} The authors are grateful to Davide Guzzetti, Gergely Harcos, Jianxun Hu, Maxim Kontsevich, Bal\'azs Szendr\H{o}i and Aleksey Zinger for helpful comments and discussions. \'A.~Gy.~was supported by the János Bolyai Research Scholarship of the Hungarian Academy of Sciences and by the European Union's Horizon 2020
research and innovation programme under the Marie Sk\l odowska-Curie grant
agreement No.\ 891437. 
Sz.~Sz. benefited of support from the grants K146401 and KKP144148 of the National Research, Development and Innovation Office of Hungary. 

\section{Notation and background}
\label{sec:notback}

\subsection{Gromov--Witten invariants}
Let $X$ be a nonsingular projective surface, $\beta \in H_2(X,\QQ)$ and $k$ be a positive integer. Denote by $\overline{\mathcal{M}}_k(X,\beta)$ the moduli space of genus zero $k$-pointed stable maps $f:C \to X$ with the homology class $f_{\ast}[C]=\beta$. Denote by $[\overline{\mathcal{M}}_k(X,\beta)]^{\mathrm{Vir}}$ the virtual fundamental class~\cite{behrend1997intrinsic} of 
$\overline{\mathcal{M}}_k(X,\beta)$. 
To each marked point there corresponds an evaluation map
\[ \mathrm{ev}_i : \overline{\mathcal{M}}_k(X,\beta) \to X, \quad 1 \leq i \leq k.\]
Given cohomology classes $\beta_1,\dots,\beta_k \in H^{\ast}(X)$, the associated (genus zero) Gromov-Witten invariant is defined as
\[ I_{\beta}(\beta_1\dots\beta_k)= \int_{[\overline{\mathcal{M}}_k(X,\beta)]^{\mathrm{Vir}}} \prod_i \mathrm{ev}_i^{\ast}(\beta_i). \]

Let $B \subset H_2(X,\QQ)$ be the effective cone, i. e. the semigroup of 
non-negative linear combinations of classes of algebraic curves on $X$ with 
rational coefficients. 
It is known that $I_{\beta}(\beta_1,\dots,\beta_k)=0$ for any $\beta \not\in B$,
because then $\overline{\mathcal{M}}_k(X,\beta)$ is empty.

Let $T_0 = 1 \in H^0(X,\QQ)$, $T_1,\dots, T_m$ be a basis of $H^2(X,\QQ)$, and $T_{m+1} = T_p \in H^4(X,\QQ)$ be the (Poincar\'e dual of the) class of a point. Denote by $T_i^{\vee}$ the corresponding elements of the dual basis: $T_i (T_j^{\vee})=T_j^{\vee} (T_i)=\delta_{ij}$. For variables $t_0, \, q_1, \, \dots, q_m, \, t_p$ (which will be also abbreviated as $q,t$), the (genus zero) 
Gromov--Witten potential is defined as the formal Laurent series
\[ F(q,t)=\sum_{\substack{n \geq 0\\ \beta \in B\setminus \{0\}} } I_{\beta}(T_p^{n})q_1^{\int_\beta T_1} \cdots q_m^{\int_\beta T_m}\frac{t_p^{n}}{n!}   \]
in the ring $\Q \llbracket q,t \rrbracket [q^{-1}]$. 
Notice that there may exist negative exponents of $q$. However, as we will compute explicitly, for surfaces the exponents are bounded below, therefore $F$ belongs to the given ring.

Let \[ \partial_i = 
\begin{cases}
q_i \frac{\partial}{\partial q_i} & i \not\in \{0,p\} \\
\frac{\partial}{\partial t_i} & i \in \{0,p\}
\end{cases}
\]
and denote $F_{ijk}=\partial_i\partial_j\partial_k F$. 
\begin{definition}\label{def:quantum_product}
The \emph{quantum product} of $T_i$ and $T_j$ is defined as
\[ T_i \ast T_j= T_i \cdot T_j + \sum_{e,f}F_{ije}g^{ef}T_f\]
where $(g_{ij})$ is the matrix of intersection numbers $(T_i \cdot T_j)$ and $(g^{ij})$ is its inverse. 
The \emph{small quantum product} is the restriction of the quantum product to the subspace $t_0 = 0 = t_p$.
\end{definition} 
The quantum product induces a $\Q \llbracket q,t \rrbracket [q^{-1}]$-algebra structure on the free $\Q \llbracket q,t \rrbracket [q^{-1}]$-module generated by $T_0,\dots,T_p$.

\subsection{The Dubrovin connection}
\label{subsec:dubrovin}

It was observed by Dubrovin \cite{dubrovin1996geometry} that the quantum product naturally gives rise to a meromorphic connection, which is part of a Frobenius 
manifold structure on $H^*(X,\CC )$ for the nonsingular pairing given by 
cup product. 

Let 
\[\tau=t_0(\tau)T_0+\sum_{i} q_i(\tau)T_i + t_m(\tau)T_m  \in H^{\ast}(X, \C).\] 
\begin{definition}\label{ref:Dubrovin_connection}
The \emph{Dubrovin} (or \emph{quantum}) \emph{connection} on the trivial bundle 
\[\mathcal{H}_{\tau} \coloneqq H^{\ast}(X,\Q) \times \Spec \C[u,u^{-1}] 
\to \Spec \C[u,u^{-1}]\] 
is the meromorphic flat connection
\[\nabla_{\frac{\partial}{\partial u}}^{(\tau)}=\frac{\partial}{\partial u} + \frac{1}{u^2}K^{(\tau)}+\frac{1}{u}G\]
on $\mathcal{H}_{\tau}$. 
Here
$K^{(\tau)}$ is the operator of taking quantum product $*$ 
(see Definition~\ref{def:quantum_product}) with the (class of the) 
\emph{Euler vector field}
\begin{equation} 
\label{eq:euler}
c_1(T_{X}) +2t_0T_0-2t_pT_p,
\end{equation}
and the grading operator $G$ is defined on a homogeneous component of the cohomology ring as
\[ G\vert_{H^d(X)}=\frac{d-2}{2} \cdot \mathrm{Id}_{H^d(X)}\]
for any $0 \leq d \leq 4$. 
\end{definition}

Suppose that the potential $F$ converges in $q,t$ or, equivalently, in $\tau$ to an analytic function on some domain. We will see in Lemma~\ref{lem:ruledGWan} and Theorem~\ref{thm:appmain} below that this assumption always holds for nonsingular projective surfaces, which will be the only case we consider. The family of meromorphic connections $\nabla^{(\tau)}$ as $\tau$ ranges over 
$H^{\ast}(X, \C)$ form a flat connection $\nabla$ over the base  
$\Spec \C \{ q, t \} [u^{\pm},q^{-1}]$ (it is an isomonodromic 
family~\cite{dubrovin1998geometry}). 
To make this more precise, let $\mathcal{H}$ stand for the trivial bundle 
\[H^{\ast}(X,\C) \times \Spec \C \{ q,t\} [q^{-1},u,u^{-1}] \to \Spec \C \{ q,t\} [q^{-1},u,u^{-1}]\]  
with fibers $H^{\ast}(X,\C)$. 
In this case $q, t$ and $u$ give a full set of coordinates on the base space. Hence, $\frac{\partial}{\partial q}$, $\frac{\partial}{\partial t}$ and $\frac{\partial}{\partial u}$
give a global basis of the tangent bundle of the base. 
The formula for the extended connection $\nabla$ in the remaining directions is
\[
\begin{gathered}
\nabla_{\frac{\partial}{\partial t_i}}=\frac{\partial}{\partial t_i} + \frac{1}{u}A_i \\ 
\nabla_{\frac{\partial}{\partial q}}=\frac{\partial}{\partial q} + \frac{1}{uq}A
\end{gathered}
\]
where $A_i$ is the operator of taking quantum product with $T_i$ and $A$ is the operator of taking quantum product with $c_1(\mathcal{O}(1))$ (it is also required that $T_i=c_1(\mathcal{O}(1))$ for some $i$). 

Associativity of the quantum product $*$ is equivalent to the flatness of the above connection. 
On the other hand, associativity is also equivalent with a certain partial differential equation satisfied by the potential, called 
Witten--Dijkgraaf--Verlinde--Verlinde (WDVV)
equation~\cite{kontsevich1994gromov, crauder1995quantum}.

\subsection{GW invariants of blow-ups}

Let now $r \geq 0$ and $X_{(r)}$ be the blow-up of $X$ in $r$ generic points. 
As above, let $T_0 = 1, T_1,\dots, T_m, T_p$ be a basis of $H^{\ast}(X,\QQ)$, where $T_1,\dots, T_m$ is a basis of $H^2(X)$ and $T_p \in H^4(X)$ is the (Poincar\'e dual of the) class of a point. Let 
\[ b: X_{(r)} \to X \]
be the blow-up map, and denote by $E_1,\dots,E_r$ the (dual) classes of the exceptional divisors. By an abuse of notation, we will write $T_0, T_1,\dots, T_m, T_p \in H^*(X_{(r)},\QQ)$ for the pullbacks under $b^{\ast}$ of the above classes. Let moreover $T_{m+i}$, $1 \leq i \leq r$ be Poincar\'e dual cohomology classes to the exceptional curves. In particular, $T_1,\dots, T_m, T_{m+1},\dots,T_{m+r}$ is a basis of $H^2( X_{(r)},\QQ)$. 
For an $r$-tuple $\alpha=(a_1,\dots,a_r)$ of integers, denote by $(\beta,\alpha)$ the homology class \[\beta-\sum_{i=1}^ra_iE_i\]
where $\beta$ is the pullback of a homology class of $X$. For any $1 \leq i \leq r$, $[i]$ will denote the $r$-tuple $\alpha$ that has 1 at the $i$-th entry and 0 everywhere else.

Let $\overline{\mathcal{M}}_{0,n}(X_{(r)},(\beta,\alpha))$ be the moduli space of stable $n$-pointed genus 0 maps with image class $(\beta,\alpha)$ on $X_{(r)}$ \cite[Theorem 1]{fulton1996notes}. Writing 
\[|\alpha|=\sum_i a_i,\]
the expected dimension of $ \overline{\mathcal{M}}_{0,0}( X_{(r)},(\beta,\alpha))$ is \cite[Theorem 2]{fulton1996notes}
\[ n_{\beta,\alpha}\coloneqq \int_{(\beta,\alpha)} c_1( X_{(r)})-1 =\int_{\beta} c_1(X)-|\alpha|-1.\] 
Let
\[ N_{\beta,\alpha}\coloneqq I_{(\beta,\alpha)}(T_p^{n_{\beta,\alpha}})\]
be the Gromov--Witten invariant for the appropriate point class on $ X_{(r)}$. 
When $\alpha$ is empty, that is, when considering GW invariants of $X$, we will just write $n_{\beta}$ and $N_{\beta}$.
For $\alpha=(a_1,\dots,a_r)$, we will write $(\alpha,0) = (a_1,\dots,a_r,0)$ and $(\alpha,1) = (a_1,\dots,a_r,1)$.

\begin{theorem}[{\cite{gottsche1998quantum, hu2006quantum}}] 
\label{thm:gwburules}
\begin{enumerate}
\item The numbers $N_{\beta,\alpha}$ satisfy the following properties. \label{thm:gwburules_1} 
\begin{enumerate}
    \item $N_{\beta,\alpha}=N_{\beta,(\alpha,0)}$ \label{thm:gwburules_a}
    \item $N_{0,\alpha}=1$ if $\alpha=-[i]$ for some $1 \leq i \leq r$, and 0 for any other $\alpha$ \label{thm:gwburules_b}
    \item $N_{\beta,\alpha}=0$ if $\beta$ is effective and any of the $a_i$ is negative \label{thm:gwburules_c}
    \item If $n_{\beta,\alpha} > 0$, then $N_{\beta,\alpha}=N_{\beta,(\alpha,1)}$. \label{thm:gwburules_d}
\end{enumerate}
\item The numbers $N_{\beta,\alpha}$ can be determined by a recursive algorithm starting from the ones given by (a) and (b) from part~\eqref{thm:gwburules_1}. 
\end{enumerate}
\end{theorem}
The (genus 0) Gromov--Witten potential simplifies according 
to~\cite[Page 8]{gottsche1998quantum} as  
\[F(q,t)=\sum_{(\beta,\alpha)}N_{\beta,\alpha}q^{\beta} q^{\alpha} \frac{t_p^{n_{\beta,\alpha}}}{n_{\beta,\alpha}!}\]
where 
\[
q^{\beta}=q_1^{\int_\beta T_1}\dots q_m^{\int_\beta T_m}, \quad q^{\alpha}= q_{m+1}^{a_1} \dots q_{m+r}^{a_r}
\]
and the sum is taken over classes $(\beta,\alpha) \neq 0$ satisfying $n_{\beta,\alpha} \geq 0$. 

\begin{lemma}
\label{lem:gij}
\begin{equation*} 
g_{ X_{(r)}}^{ij}=\begin{cases}
1 & \textrm{ if } (i,j) \in \{(0,p),(p,0)\}, \\
g_{X}^{ij} & \textrm{ if } i,j \in \{1,\dots,m\}, \\
-1 & \textrm{ if } (i,j) \in \{(m+1,m+1),\dots,(m+r,m+r)\}, \\
0 & \textrm{ otherwise}
\end{cases} \end{equation*}
\end{lemma}
\begin{proof}
This follows from the fact that $r$ generic points were blown-up.
\end{proof}

With these notations, Definition~\ref{def:quantum_product} simplifies as follows.

\begin{corollary}\label{cor:quantprodrec}
The quantum product on $H^*(X, \C)$ is given by the expression  
\begin{equation*} 
T_i \ast T_j = (T_i \cdot T_j)T_p + \sum_{k,l=1}^mF_{ijk}g^{kl}T_l-\sum_{e=m+1}^{m+r}F_{ije}T_e + F_{ijp}T_0.\end{equation*}
\end{corollary}

\subsection{Consequences of the WDVV relations}
\label{subsec:wdvv}

When applying induction in our arguments, we will set $r=1$. In this case, for convenience, the class $(\beta,\alpha)$ will be also written as $(\beta,a)$.

Let the symbol $\vdash (\beta,a)$ denote the set of pairs $((\beta_1,a_1), (\beta_2,a_2))$ satisfying
\begin{enumerate}
    \item[(i)] $(\beta_1,a_1), (\beta_2,a_2) \neq 0$
    \item[(ii)] $(\beta_1,a_1)+ (\beta_2,a_2)= (\beta,a)$
    \item[(iii)] $n_{\beta_1,a_1},n_{\beta_2,a_2} \geq 0$ 
\end{enumerate}
Let $\vdash (\beta,a) \neq 0$ denote the subset of $\vdash (\beta,a)$ for which $\beta_1 \neq 0, \, \beta_2 \neq 0$. The following two recursive relations were obtained in \cite[Theorem 3.3]{hu2006quantum} using the WDVV equations. First, if $n_{\beta,a} \geq 3$ and $g_{ij}\neq 0$ for some $i,j \in \{1,\dots,m\}$, then
\begin{equation}
\label{eq:Rm}
\begin{multlined}
N_{\beta,a}= \frac{1}{g_{ij}}\sum_{\vdash (\beta,a) \neq 0} N_{\beta_1,a_1}N_{\beta_2,a_2}\left(\sum_{a,b=1}^mT_a(\beta_1)g^{ab}T_b(\beta_2)-a_1a_2\right) \\
\cdot \left[T_i(\beta_1)T_j(\beta_2)\binom{n_{\beta,a}-3}{n_{\beta_1,a_1}-1}-T_i(\beta_1)T_j(\beta_1)\binom{n_{\beta,a}-3}{n_{\beta_1,a_1}}\right]
\end{multlined}
\end{equation}
Second, if $n_{\beta,a} \geq 0$, then
\begin{equation}
\label{eq:Rk}
\begin{multlined}
T_i(\beta)T_j(\beta)a N_{\beta,a}= (T_i(\beta) T_j(\beta)-g_{ij}(a-1)^2))N_{\beta,a-1} \\
+ \sum_{\vdash (\beta,a-1) \neq 0} N_{\beta_1,a_1}N_{\beta_2,a_2}\left(\sum_{a,b=1}^mT_a(\beta_1)g^{ab}T_b(\beta_2)-a_1a_2\right) \\
\cdot \left(T_j(\beta_1)T_i(\beta_2)a_1a_2-T_i(\beta_1)T_j(\beta_1)a_2^2\right)\binom{n_{\beta,a-1}-1}{n_{\beta_1,a_1}}
\end{multlined}
\end{equation}

\section{The operator $K$}
\label{sec:opK}

\subsection{Structure of $K$}
\label{subsec:Kstr}

The class $T_0$ is the identity for the quantum product. Hence, the term $2t_0 T_0$ only contributes the diagonal matrix $2t_0I$ to $K$, which corresponds to a shift of the spectrum. We will therefore assume $t_0=0$ throughout the paper.

We will carry out an asymptotic analysis of the spectrum of the Dubrovin connection as $q\to 0$ and $t_p$ is fixed such that $|t_p| \ll \infty$.
For this we will need the leading (lowest order) terms in the connection matrix with respect to the variables $q$. We will first give a recursive formula for these terms. In this section we will assume that $r=1$, and investigate the Gromov-Witten invariants of $X_1$ when compared to those of $X$. For convenience, in this section we will denote $q_{m+1}$, the variable corresponding to the exceptional divisor $E$, by $q_e$.

We will now compare the operator $K$ with the operator $\overline{K}$ of $X$. Recall as well that for an integer $a$, $(\beta,a)$ denotes the homology class \[\beta - aE\] where $\beta$ is the pullback of a homology class of $X$. 

Blowing up a smooth variety in a smooth subvariety of codimension $n$ adds $-(n-1)E$ to its first Chern 
 class~\cite[Section~4.6]{griffiths1978principles}. As we consider surfaces blown up in a point, 
\begin{equation}
\label{eq:c1bu}
c_1(X_1)= c_1(X)-E.
\end{equation}
Hence, we have that
\[ c_1(X_1) \cdot (\beta, a) = c_1(X)\cdot\beta -a.\]

When considering $F_{ijk}$ for a triple of indices $i,j,k$ 
let
\[ \epsilon=\epsilon(i,j,k)\coloneqq \delta_{ip}+\delta_{jp}+\delta_{kp} \in \{0,1,2,3\}.\]
Let 
\[ T'_{i}(\beta,a) \coloneqq \begin{cases}
T_i(\beta) & \textrm{if }  i \in \{1,\dots,m\} \\
a & \textrm{if }  i =e \\
1 & \textrm{if }  i = p \\
0 & \textrm{if }  i = 0
\end{cases} \]
\begin{lemma}\label{lem:potderivrec} 
\[
F_{ijk} =  q_e^{-1}\delta_{ie}\delta_{je}\delta_{ke}+\overline{F}_{ijk} +  \sum_{\substack{(\beta,a)\\ a > 0}}N_{\beta,a}T'_i(\beta,a)T'_j(\beta,a)T'_k(\beta,a)
q^{\beta}q_e^{a} \frac{t_p^{n_{\beta,a}-\epsilon}}{(n_{\beta,a}-\epsilon)!}
\]
where $\overline{F}$ is the potential of $X$. 
\end{lemma}
\begin{proof}
This follows from the definition of the potential and the rules of derivation.
\end{proof}

\begin{proposition} 
\label{prop:Kij}
For $i \neq p$ and $j\neq 0$, 
\[
\begin{aligned}
K_{ij} = & \overline{K}_{ij} -q_e^{-1}\delta_{ie}\delta_{je} \\
& +\sum_{\substack{\beta,a\\a >0}}N_{\beta,a} 
\left(\sum_{k \in \{1,\dots,m, e, p\}} g^{ki}T'_k(\beta)\right)  T'_j(\beta)(1-n_{\beta,a}+2\varepsilon)q^{\beta}q_e^{a}\frac{t_p^{n_{\beta,a}-\varepsilon}}{(n_{\beta,a}-\varepsilon)!}
\end{aligned}
\]

where $\overline{K}_{ij}$ is understood to be 0 if either $i$ or $j$ is equal to $e$, and \[\varepsilon=\varepsilon(i,j) \coloneqq \delta_{i0}+\delta_{jp} \in \{0,1,2\}\]
\end{proposition}
\begin{proof}
Rewrite \eqref{eq:c1bu} as
\[c_1(X_1)=c_1(X)-E=\sum_{l=1}^mT^{\vee}_l(c_1(X))T_l-E.\]
This gives
\[ K_{ij}=\sum_{k \in \{1,\dots,m, e, p\}} 
g^{ki}\left( \sum_{l=1}^mT^{\vee}_l(c_1(X))F_{ljk}-F_{ejk}-2t_pF_{pjk}\right)\]
Applying Lemma~\ref{lem:potderivrec} then yields 
\[\begin{gathered}
K_{ij} =\overline{K}_{ij} -q_e^{-1}\delta_{ie}\delta_{je} \\
+\sum_{\substack{\beta,a\\a >0}}N_{\beta,a}T'_j(\beta,a) q^{\beta}q_e^{a} \frac{t_p^{n_{\beta,a}-\varepsilon}}{(n_{\beta,a}-\varepsilon)!}\\ \cdot \sum_{k \in \{1,\dots,m, e, p\}
}g^{ki}\left(\sum_{l=1}^m T^{\vee}_l(c_1(X))T_l(\beta)T'_k(\beta,a)-T_e'(\beta,a)T_k'(\beta,a)-2(n_{\beta,a}-\varepsilon)T_k'(\beta,a)\right)
\end{gathered}\]
As
\[\sum_{l=1}^m T^{\vee}_l(c_1(X))T_l(\beta) -T_e'(\beta,a) = \int_{\beta}c_1(X)-a=n_{\beta,a}+1,\]
we get the statement. Note that the constant $\epsilon(i,j,k)$ of Lemma~\ref{lem:potderivrec} in the above proof was replaced by $\varepsilon(i,j)$ due to Lemma~\ref{lem:gij}.
\end{proof}

\begin{lemma}\label{lem:Ki0Kpi} For $i \in \{0,1,\dots,m,e,p\}$,
\[K_{i0} = T_i^{\vee}(c_1(X_1)),\quad K_{pi} = T_i \cdot c_1(X_1)\quad\textrm{ and }\quad K_{p0} = -2 t_p.\]
In particular,
\[ K_{i0}=\begin{cases}
\overline{K}_{i0}, & \textrm{ if } i \neq e \\
-1,& \textrm{ if } i = e
\end{cases}
\quad \textrm{and}\quad 
K_{pi}=\begin{cases}
\overline{K}_{pi}, & \textrm{ if } i \neq e \\
1,& \textrm{ if } i = e.
\end{cases}
\]

\end{lemma}
\begin{proof}
This follows from the fact that $T_0$ is the identity for the quantum product \cite[Lemma~3.1]{crauder1995quantum} and from Corollary~\ref{cor:quantprodrec}.
\end{proof}

We will say that $K_{i0}$, $K_{pi}$, $i \in \{0,1,\dots,m,e,p\}$ are the trivial entries of $K$.

Let us fix some $t_p \neq 0$. 
\begin{defn}
The \emph{degree} or, interchangeably, the \emph{order} of a summand in $K_{ij}$ is the sum of the exponents of the $q$-variables.
We will say that $(\beta,a)$ \emph{appears} or, interchangeably, \emph{occurs} in $K_{ij}$ (or just shortly, at $ij$)
if $q^{\beta}q_e^a$ occurs in $K_{ij}$ as a summand with nonzero coefficient. 
If moreover the order of $q^{\beta}q_e^a$ is minimal among the summands of $K_{ij}$, 
we say that $(\beta,a)$ is \emph{minimal} at $ij$. We denote by
\[\min \deg K_{ij}\] the minimal degree of the summands of $K_{ij}$.
\end{defn} 

With an argument similar to that of Proposition~\ref{prop:Kij}, one can also show the following.
\begin{lemma} For $i \neq p$ and $j \neq 0$
\[
\overline{K}_{ij} =
\sum_{\beta}N_{\beta}  
\left(\sum_{k \in \{1,\dots,m,p\}} g^{ki}T'_k(\beta)\right)\left(1-n_{\beta}+2\varepsilon\right) T'_j(\beta)q^{\beta}\frac{t^{n_{\beta}-\varepsilon}}{(n_{\beta}-\varepsilon)!}
\]
\end{lemma}

\subsection{Symmetry of the exceptional curves}
\label{subsec:symm}
 

We now return to the case of general $r$ but we suppose from here on that we start from surface $X_{\min}$, which is a minimal model in the sense that it does not contain any $(-1)$-curve. Its blow-up in $r$ generic points is denoted by $X_r$.

As above, $X_r$ is obtained from $X_{\min}$ by blowing-up $r$ (generic) points via the map
\[ b: X_r \to X_{\min}, \]
the classes $T_0, T_1,\dots, T_m, T_{m+1},\dots,T_{m+r}, T_p$ form a basis of $H^\ast(X_r)$, and we consider the invariants 
\[ N_{\beta,\alpha}= I_{\beta,\alpha}(T_p^{n_{\beta,\alpha}})\]
for homology classes $(\beta,\alpha)$ with $\beta \in H^{\ast}(X_{\min})$ and $\alpha=(a_1,\dots,a_r)$.

For any permutation $\sigma \in S_r$ of the exceptional divisors one always has
\[ N_{\beta,\alpha}=N_{\beta,\sigma(\alpha)} \]
 by \cite[Theorem 3.3]{hu2006quantum}.
This implies the following.
\begin{lemma} 
\label{lem:excperm}
Let $i, j\in \{1,\dots,m\}$.
If 
\[ q^{\beta} q_{m+1}^{a_1}\dots q_{m+r}^{a_r} \frac{t_p^{n_{\beta,\alpha}}}{(n_{\beta,\alpha})!} \] is a summand of $K_{ij}$, then the terms
\[ q^{\beta} \left(\sum_{\sigma \in S_r} q_{m+1}^{\sigma(a_1)}\dots q_{m+r}^{\sigma(a_r)}\right) \frac{t_p^{n_{\beta,\alpha}}}{(n_{\beta,\alpha})!} \]
are all summands of $K_{ij}$, and they all have the same coefficient.
\end{lemma}

\section{The Newton polygon of the blow-up}
\label{sec:newtonpoly}

\subsection{The Newton polygon}
\label{subsec:newton_poly}

For a smooth surface $X$ denote by $\chi_{X}(\lambda)$ the characteristic polynomial in the indeterminate $\lambda$ of the operator $K$ associated with $X$.
It has coefficients in $\mathbb{C}\llbracket q_1,\ldots ,q_{m+r}, t_p \rrbracket [q_1^{-1},\dots,q_{m+r}^{-1}]$ (as before, we fixed $t_0=0$). 
From now on we assume that $q \neq 0$ and 
\[
\begin{aligned}
    q_1 & =  \nu_1 q \\ 
    & \vdots  \\
    q_{m+r} &=  \nu_{m+r} q \\
\end{aligned}
\]
for some $\nu_1, \ldots ,\nu_{m+r}\in \mathbb{C} \setminus \{ 0\}$. We will treat $t_p$ as a constant.
In this way, the coefficients of $\chi_{X}(\lambda)$ may be considered as elements of  $\mathbb{C}\llbracket q \rrbracket[q^{-1}]$. 
Notice that as we may rescale $q$, the $(m+r)$-tuple $\nu_1, \ldots ,\nu_{m+r}$
is only defined up to scale. That is, we actually consider 
\[  [\nu_1 \colon \cdots \colon \nu_{m+r}] \in  \CP{m+r-1}  \]
\begin{defn}
The \emph{Newton pairs} (sometimes, Puiseux pairs) of $\chi_{X}(\lambda)$ are the
lattice points  $(x,y)\in\ZZ^2$ such that the coefficient of the monomial 
$\lambda^x q^y$ 
is non-zero. 
The \emph{Newton polygon} of the characteristic polynomial $\chi_{X}(\lambda)$ 
is the lower convex hull of the set of Newton pairs in the $(x,y)$-plane, i. e. 
the smallest convex set containing the rays parallel to the positive $y$-axis 
emanating from the Newton pairs. 
\end{defn}

The lower boundary of the Newton polygon is a broken straight line that we will often 
identify with the polygon itself. We will freely use the terminology of plane 
co-ordinate geometry for broken line segments such as slope, salient 
vertex, etc. 

\begin{theorem}
\label{mainthm1}
Let  $X_{\min}$ be a minimal model and $X=X_r$ is its $r$-fold blow-up in 
generic points. 
Then, for all $ [\nu_1 \colon \cdots \colon \nu_{m+r+1}] \in \CP{m+r}$ 
the Newton polygon of $\chi_{X_r}(\lambda)$ is obtained 
by translating the Newton polygon of $\chi_{X_{\min}}(\lambda)$ by the vector $(0,-r)$, 
and extending it by a segment of slope $1$ and length $r$ on its right. 
\end{theorem}

\begin{proof}

The proof contains three steps. First, we show that there do exist Newton pairs as in the statement. Second, that any further Newton pair lies on or above the diagram determined by the pairs of the statement. Third, that the monomials corresponding to the salient vertices of the boundary are unique.

For the first one we will use induction. We will show that the Newton polygon of $\chi_{X_r}(\lambda)$ is obtained 
by translating the Newton polygon of $\chi_{X_{r-1}}(\lambda)$ by the vector $(0,-1)$, 
and extending it by a  further vertex $(1+\mathrm{deg }\,\chi_{X_{r-1}}, 0)$ (on its right). 
 
From Proposition~\ref{prop:Kij} it follows that each Newton pair of 
$\chi_{X_{r-1}}(\lambda)$ translated by $(0,-1)$ 
appears as a Newton pair of $\chi_{X_{r}}(\lambda)$. 
Namely, in the $(m+r,m+r)$ entry of $K$ there appears a minimal term $q_{m+r}^{-1}$. Hence, for every minimal monomial 
\[
\begin{multlined}
    c_{x,y_1,\ldots , y_{m+r-1}} \lambda^x q_1^{y_1}\cdots q_{m+r-1}^{y_{m+r-1}}
    \\ =
    c_{x,y_1,\ldots , y_{m+r-1}} \nu_1^{y_1}\cdots \nu_{m+r-1}^{y_{m+r-1}} \nu_{m+r+1}^{y_{m+r+1}}
    \lambda^x q^{y_1 + \cdots +y_{m+r-1}}
\end{multlined}
\]
of $\chi_{X_{r-1}}(\lambda)$ with $c_{x,y_1,\ldots , y_{m+r-1}} \neq 0$, 
we get a monomial 
\begin{equation*}
c_{x,y_1,\ldots , y_{m+r-1}}  \nu_1^{y_1}\cdots \nu_{m+r-1}^{y_{m+r-1}}
   \lambda^x \nu_{m+r}^{-1} q^{y_1 + \cdots y_{m+r-1}-1}
\end{equation*}
in $\chi_{X_{r}}(\lambda)$.
In view of our assumption $\nu_{m+r}\neq 0$, this shows the first statement.

For the second and third statements we will give a case-by-case proof using the following classification of minimal surfaces. As $X_{\min}$ is minimal, it falls into one of the following cases \cite[Theorem~1.29]{kollar1998birational}:
\begin{enumerate}
    \item $-c_1(X_{\min})$ is nef
    \item $X_{\min}$ is a minimal ruled surface over some curve $\Sigma_g$
    \item $X_{\min} \simeq \CP2$
\end{enumerate}
\end{proof}

\subsection{Case 1: $-c_1(X_{\min})$ is nef}

Theorem~\ref{mainthm1} 
 in this case is deduced from the following explicit calculations.

\begin{lemma}
\label{lem:c1antinef0}
Let $\beta \in H^2(X_{\min})$ such that $n_{\beta}\geq 0$. Then $N_{\beta}=0$.
\end{lemma}
\begin{proof}
If $N_{\beta} \neq 0$, then there is a (not necessarily irreducible or reduced) curve $C \subset X_{\min}$ in the class $\beta$ for which
\[\int_C c_1( X_{\min}) > 0\]
But this contradicts that $-c_1(X_{\min})$ is nef.
\end{proof}

\begin{lemma}
\label{lem:c1antinef}
Let $(\beta,\alpha) \in H^2(X_{r})$ such that $n_{\beta,\alpha} \geq 0$. Then
\[N_{\beta,\alpha}=
\begin{cases}
1 & \textrm{ if } \alpha = -[i] \quad \textrm{for some } 1 \leq i \leq r \\
0 & \textrm{ otherwise}
\end{cases}\]
\end{lemma}
\begin{proof}
If $\alpha=0$, we are done by Lemma~\ref{lem:c1antinef0}.

Suppose that $\alpha \neq 0$ and $N_{\beta,\alpha} \neq 0$. As in the previous proof, there exists then a curve $C \subset X_r$ in the class $(\beta,\alpha)$ for which
\begin{equation} 
\label{eq:intpos}
\int_C c_1(X_r) > 0.\end{equation}
As \[c_1(X_r)=p^{\ast}(c_1(X_{\min}))-E_1-\dots-E_r,\] we have that
\[ \int_C c_1(X_r) = \int_C p^{\ast}(c_1(X_{\min}))-E_1\cdot C-\dots-E_r\cdot C.\]
On one hand, we can apply the projection formula \cite[Proposition~2.5~(c)]{fulton2017intersection} on the proper map $p$ to obtain that
\[\int_C p^{\ast}(c_1( X_{\min}))=\int_{p_{\ast}(C)}c_1(X_{\min}).\]
On the other hand, the entries of $\alpha$ must be all nonnegative by Theorem~\ref{thm:gwburules}~\eqref{thm:gwburules_c} because $\beta \neq 0$. Suppose that there is at least one positive entry in $\alpha$.
Then
\[ 1 \leq E_1\cdot C+\dots+E_r\cdot C.\]

If $C$ is not contained entirely in the exceptional locus, then because $p$ is birational, the image
\[p_{\ast}(C) \subset X_{\min}\] 
is a (not necessarily rational) curve. As $-c_1( X_{\min})$ is nef, we must have
\[ c_1( X_{\min})(p_{\ast}(C)) \leq 0. \]
Hence, $\alpha \neq 0$ would imply \[\int_C c_1(X_r)< 0.\] 
But this would contradict \eqref{eq:intpos} when $N_{\beta,\alpha} \neq 0$. 

If $C$ is contained in the exceptional locus of $p$, then by Theorem~\ref{thm:gwburules} it must be of class $-[i]$ 
for a single exceptional curve $E_i$. 
\end{proof}

\begin{proposition} When $-c_1(X_{\min})$ is nef,
\label{prop:c1antinef}
\[ K_{ij} = \begin{cases}
q_i^{-1} & \textrm{if } i=j \in \{ m+1,\dots,m+r \} \\
0 & \textrm{otherwise} 
\end{cases}\]
 for $i \neq p$ and $j \neq 0$. In particular, $K$ is a lower triangular matrix with diagonal \[(0,\dots, 0,q_{m+1}^{-1},\dots,q_{m+r}^{-1},0).\]
\end{proposition}

\subsection{Case 2: $X_{\min}$ is a minimal ruled surface}

Let $U$ be a vector bundle of rank two over some curve $\Sigma_g$ of genus $g$ such that $ X_{\min}=\mathbb{P}(U)$ with 
\[\pi: X_{\min}=\mathbb{P}(U) \to \Sigma_g \]
the projection. The second homology of $X_{\min}$ is generated by $c=c_1(\mathcal{O}_{ X_{\min}}(1))$ and $f=[F]$ where $F$ is a fiber of $\pi$.
If we denote $u= \mathrm{deg}\, U \coloneqq \mathrm{deg} (\det\, U)$, these classes intersect each other as
\[ f^2=0, \quad c\cdot f=1, \quad c^2=u.\]
To keep track of their different roles, only in this section we will write $q_f$ and $q_c$ for the variables corresponding to $f$ and $c$ instead of numbering them.
The inverse of the matrix (written in the basis $f,c$)
\[
\begin{pmatrix}
0 & 1 \\
1 & u
\end{pmatrix}
\]
is
\[
\begin{pmatrix}
- u & 1 \\
1 & 0
\end{pmatrix}
\]
There are two cases to consider when describing the cone of effective curves on $ X_{\min}$.

First, $U$ can be unstable. This means that there exists a line bundle quotient $A$ of degree \[a = \mathrm{deg}\,(A) \leq \frac 12 u.\] 
Then
\[ \mathbb{P}(A) \subset \mathbb{P}(U)= X_{\min} \]
is an effective curve in the class $af+c$, and the ray spanned by it bounds the cone of curves \cite[Chapter 1, 1.5.A]{lazarsfeld2017positivity}. 
The other boundary of the cone of curves is the ray spanned by $f$.

Second, when $U$ is not unstable, then it is semistable. Then the cone of effective curves is bounded by the rays spanned by $c$ and $f$.

It is known \cite[V.2.10]{hartshorne1977algebraic} that \[c_1( X_{\min})=2c+(2-2g-u)f\] 
For a class $\beta=bf+dc$ the expected dimension is therefore 
\[
\begin{aligned}
n_{\beta} & =c_1(X_{\min})\cdot (bf+dc)-1 \\
& =(2-2g+u)d+2b-1 
\end{aligned}
\]

\begin{example}
If $u=0$, the expected dimension $n_{\beta}$ is nonnegative when $b > d(g-1)$. The lines $b = d(g-1)$ for small values of $g$ look as follows.
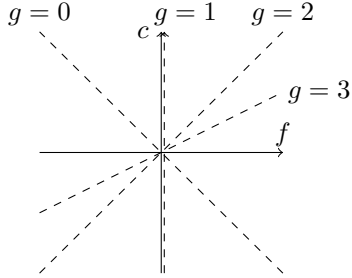
\begin{center}
\begin{tikzpicture}[scale=0.8]
\draw[->] (-2,0) -- (2,0);
\draw[->] (0,-2) -- (0,2);
\node at (-0.3,2) {$c$};
\node at (2,0.3) {$f$};
\draw[dashed] (2,-2) -- (-2,2);
\node at (-2,2.3) {$g=0$};
\draw[dashed,->] (0.05,-2) -- (0.05,2);
\node at (0.4,2.3) {$g=1$};
\draw[dashed] (-2,-2) -- (2,2);
\node at (2,2.3) {$g=2$};
\draw[dashed] (-2,-1) -- (2,1);
\node at (2.6,1) {$g=3$};
\end{tikzpicture}
\end{center}
The open half-planes to the right of the lines $b = d(g-1)$ correspond to classes with nonnegative expected dimension.
\end{example}

\subsubsection{Higher genus}

If $L$ is a line bundle on $\Sigma_g$, then $\mathbb{P}(U) \simeq \mathbb{P}(U\otimes L)$ \cite[Chapter V, Proposition~2.2]{hartshorne1977algebraic} and 
\[ \mathrm{deg}\,(U\otimes L)=\mathrm{deg}\,(U)+2\mathrm{deg}\,(L). \]
By replacing $U$ with some $U\otimes L$ if necessary,
we will assume that $u \neq 0$ 
and $u \neq 2-2g$. 

Recall from Section~\ref{subsec:newton_poly} that the variable $t_p$ is considered to be a constant. We will say that a summand of an entry of $K$ is \emph{of lowest order}, if the sum of the exponents of the $q$-variables is the lowest in it among all summands of that entry. We will show the following.

\begin{proposition}
\label{prop:ggeq0}
The lowest order terms of $K$ on $X_r$, the blow-up of a ruled surface over a curve of positive genus  in $r$ points are as follows:
\begin{enumerate}
    \item $2q_f$ at $0f$ and $cp$ \label{prop:ggeq0_0f}
    \item $q_ft_p$ at $cf$ \label{prop:ggeq0_cf}
    \item $q_fq_j$ at $jf$ and  $cj$, $3\leq j \leq r+2$ \label{prop:ggeq0_cj}
    \item $\pm \frac{u}{2}q_f^2q_j^2$ at $0j$ and $jp$, $3\leq j \leq r+2$ \label{prop:ggeq0_jp}
    \item $-q_j^{-1}$ at $jj$, $3\leq j \leq r+2$ \label{prop:ggeq0_jj}
    \item $0c$, $fc$, $fp$, $0p$, $fj$ and $jc$, $3\leq j \leq r+2$ are all zero \label{prop:ggeq0_0c}
\end{enumerate}
\end{proposition}

To prove Proposition~\ref{prop:ggeq0}, we first need some auxiliary results. 
\begin{lemma}
\label{lem:K0pzero}
For $g >0$, all classes $(\beta,\alpha)=(bf+dc,\alpha) \in H^2(X_r)$ with $n_{\beta,\alpha}>1$ have $N_{\beta,\alpha}=0$.
\end{lemma}
\begin{proof}
GW invariants as in the claim count $n_{\beta,\alpha}$-pointed genus 0 stable maps passing through two generic points. These generic points can be chosen to lie in different fibers of the composite morphism 
\[ X_r \xrightarrow{b} X_{\min} \xrightarrow{\pi} \Sigma_g,\] and hence they have disjoint image in $\Sigma_g$. Therefore the composition of a $n_{\beta,\alpha}$-pointed stable map with $\pi \circ b$ would be surjective. In particular, the domain of the stable map must have at least one irreducible component which is nonconstant when composed with $\pi \circ b$. But such a morphism from a rational curve cannot exist when $g >0$.
\end{proof}

\begin{lemma}
\label{lem:Xrhor}
For $g >0$, all classes $(\beta,\alpha)=(bf+dc,\alpha) \in H^2(X_r)$ with $d >0$ 
have $N_{\beta,\alpha}=0$.
\end{lemma}
\begin{proof}
As $d>0$, the image of such a stable map projects surjectively onto $\Sigma_g$ under the composition $\pi \circ b$. Again, such a morphism from a rational curve cannot exist when $g >0$.
\end{proof}

\begin{lemma}
\label{lem:sigmagaizero}
For $g >0$, all classes $(\beta,\alpha) \in H^2(X_r)$ with $a_i >0$ for more than one $i$
have $N_{\beta,\alpha}=0$.
\end{lemma}
\begin{proof}
If a class $(\beta,\alpha)=(bf,\alpha)$ is such that $a_i > 0$ for more than one exceptional divisor, then the image of a stable map in this class needs to meet these divisors. Just as in the previous two Lemmas, this is impossible as the base curve has $g >0$. 
\end{proof}
\begin{lemma}
\label{lem:sigmagzero}
For $g >0$, if $N_{\beta,\alpha} \neq 0$ then $(\beta,\alpha)$ is of the form $bf-aE$ where $E$ is one of the exceptional divisors and either $0 \leq a \leq b$ or $b=0$ and $a=-1$. 
\end{lemma}
\begin{proof}
Lemmas~\ref{lem:Xrhor},~\ref{lem:sigmagaizero} imply that the image of a stable map in the class $(\beta,\alpha)$ composed with the map $b$ is a curve on $ X_{\min}$ in class $b'f$ for some $b'$.  The proper transform of a curve in the class $f$ passing through the blow-up locus  
has class $f-E$. 
This observation and Theorem~\ref{thm:gwburules} implies the statement.
\end{proof}

\begin{proof}[{Proof of Proposition~\ref{prop:ggeq0}}]
Write
\[K(t_p)= \sum_{k=0}^\infty K^{(n)}(0)\frac{t_p^n}{n!} \]
The entries of $K^{(0)}(0)$ at $ij$, $i,j\neq 0,p$, resp. at $0j$ and $ip$, $i \neq 0$, $j \neq p$ enumerate classes that have $n_{\beta,\alpha}=0$, resp. $n_{\beta,\alpha}=1$. Similarly, the entries of $K^{(1)}(0)$ at $ij$, $i,j\neq 0,p$ enumerate classes that have $n_{\beta,\alpha}=1$. By Lemma~\ref{lem:K0pzero}, every other nontrivial entry of $K^{(1)}(0)$ and also every other $K^{(n)}(0)$, $n > 1$ is zero.

By Lemmas~\ref{lem:Xrhor},~\ref{lem:sigmagaizero} and~\ref{lem:sigmagzero} 
combined with the formula for $n_{\beta,\alpha}$, 
the classes with $n_{\beta,\alpha}=0$, resp. $n_{\beta,\alpha}=1$ are of the form $bf-(2b-1)E$, resp. $bf-(2b-2)E$ for one exceptional divisor $E$ such that 
\[0 \leq 2b-1 \leq b,\] 
resp. 
\[ 0 \leq 2b-2 \leq b.\]
This has only one, resp. two solutions: $f-E$, resp. $f$ and $2f-2E$.

\eqref{prop:ggeq0_0f}: 
The class $f$ has $n_f=1$ and hence $q_f$ is a candidate to be minimal at $0f$ and $cp$. Passing through any point of $ X_{\min}$ there is a unique effective stable map whose image has class $f$: the one which maps $\CP1$ to the specific fiber of \[ \pi: X_{\min} \to \Sigma_g.\]
Hence, $N_{f}=1$, and $q_f$ is indeed one of the minimal summands at $0f$ and $cp$. In fact, for $g > 0$ it is the unique minimal summand at these entries by the above description of the effective cone. 

It follows from Part~\eqref{prop:ggeq0_0f} and Theorem~\ref{thm:gwburules}~\eqref{thm:gwburules_d} 
that $N_{f,1}=N_f=1$. 
Parts~\eqref{prop:ggeq0_cf} and~\eqref{prop:ggeq0_cj}  then follow from
Proposition~\ref{prop:Kij}.

 Applying relation~\eqref{eq:Rk} with $i=j=f$ on the class $(2f,2)$ we get that 
 \[ 2^3 \cdot N_{2f,2}=2^2\cdot \overbrace{N_{2f,1}}^{=0}+\overbrace{N_{f,1}}^{=1}\overbrace{N_{f}}^{=1}(-u)\cdot(-1) \binom{1}{1}\]
 which implies
 \[N_{2f,2} = \frac{u}{8}.\]
To get the entries at ${0j}$ and ${jp}$ this needs to be multiplied with \[ \pm T_j'(2f,2)(1+\varepsilon)=\pm 2\cdot2= \pm 4.\]

This implies Part~\eqref{prop:ggeq0_jp}.

 Part~\eqref{prop:ggeq0_jj} follows from Proposition~\ref{prop:Kij}. 
 
 Part~\eqref{prop:ggeq0_0c} follows from Lemmas~\ref{lem:K0pzero} (for $0p$) and \ref{lem:Xrhor} (for entries in 
 the row of $f$ or the column of $c$). Indeed, terms at these entries contain derivatives of 
 the potential with respect to $q_c$, and according to Lemma~\ref{lem:Xrhor} the 
 coefficients of such terms in the potential vanish. 

\end{proof}

\begin{example} 
\label{ex:SigmagK}

Let us use the ordering $0,f,c,p$ of the cohomology classes to express the matrix 
of $K$ for $X_{\min}$ a minimal ruled surface over $\Sigma_g$, $g > 0$. Then 
the minimal terms in $K$ are
\[
\begin{pmatrix}
0 & 2q_f & 0 & 0 \\
2-2g-u& 0 & 0 & 0 \\
2& q_ft_p & 0 & 2q_f \\
 -2 t_p & 2 & 2-2g+u  & 0
\end{pmatrix}
\]
The Newton pairs are therefore
\[ q^2,q\lambda^2,\lambda^4. \]
\end{example}

\begin{example} 
\label{ex:SigmagBUK}
Blowing up once the surface from Example~\ref{ex:SigmagK} 
yields
\[
\begin{pmatrix}
0 & 2q_f & 0 & \frac{u}{2}q_f^{2}q_3^{2} & 0 \\
2-2g-u& 0 & 0& 0 & 0 \\
2& q_ft_p &0  & q_fq_3 & 2q_f  \\
-1& -q_fq_3 &0  & -q_3^{-1} & -\frac{u}{2}q_f^{2}q_3^{2} \\
 -2 t_p & 2 & 2-2g+u &1 & 0
\end{pmatrix}
\]
for the minimal entries of $K$ where we used the ordering $0,f,c,3,p$ with $3$ refering to the first exceptional divisor. 
Hence, for generic $t_p$ the Newton diagram is
\[ q,\lambda^2,q^{-1}\lambda^4,\lambda^5. \]
\end{example}

\begin{lemma}
\label{lem:ruledGWan}
For $X_r$, $r \geq 0$ the genus 0 Gromov-Witten potential as well as the Dubrovin connection is a Laurent polynomial in $q,t$. 
In particular, it is $\C$-analytic.
\end{lemma}
\begin{proof}
This follows from the fact that there are only finitely many cohomology classes with nontrivial invariants on $X_r$.
\end{proof}

Proposition~\ref{prop:ggeq0} also implies the following.
\begin{lemma} 
\label{lem:rgeq1ruled}
The degree of the minimal terms of $K$ does not depend on $r$ in the following sense:
\[ \min \deg K_{ij} = \min \deg \overline{K}_{ij}, \quad i,j \in \{0,f,c,3,\dots,r+1,p\}  \]
\end{lemma}

\begin{proof}[Proof of Theorem~\ref{mainthm1} in the higher genus case]
A minimal monomial of $\chi_{X_r}$ factors either as
\[ \hat{m}_{r-1}(q,\lambda)m_{ie}m_{ej} \]
where $\hat{m}_{r-1}(q,\lambda)$ is a monomial from $\chi_{X_{r-1}}$ divided by a summand of $(K-\lambda I)_{ij}$ for some $i,j \neq e$ and $m_{ie}$, $m_{ej}$ are summands of $K_{ie}$, $K_{ej}$ respectively, or as
\[ m_{r-1}(q,\lambda)m_{ee} \]
where 
$m_{r-1}(q,\lambda)$ is a monomial from $\chi_{X_{r-1}}$ and $m_{ee}$ is a monomial from $K_{ee}-\lambda$.
Lemma~\ref{lem:rgeq1ruled} and Proposition~\ref{prop:ggeq0} implies
\begin{equation} 
\label{eq:mindegineqruled}
\min \deg\, m_{ie}m_{ej}  \geq \min \deg\,(K-\lambda I)_{ij} \quad \textrm{and}\quad \min \deg\, m_{ee} = -1, \end{equation}
which gives the second statement of Theorem~\ref{mainthm1}.

The monomial $\lambda^{4+r}$ is always the unique summand of $\chi_{X_r}$ corresponding to the lattice point $(4+r,0)$. It follows from Proposition~\ref{prop:ggeq0} that the monomials corresponding to the lattice points $(0,-r+2)$, resp. $(4,-r)$ are also unique: they are
\[ 4((2-2g)^2-u^2)q_f^2\prod_{i=3}^{r+2}q_{i}^{-1},\] resp. \[\lambda^4\prod_{i=3}^{r+2}q_{i}^{-1}.\]
As these three vertices are the unique salient vertices of the Newton polygon, 
this completes the proof (of the third statement) of Theorem~\ref{mainthm1}.
\end{proof}

\subsubsection{Genus zero}

It is known that $ X_{\min} = \mathbb{P}(U)$ is deformation-equivalent to $\mathbb{P}(U(n))$ for any twist 
\[U(n) = U\otimes \mathcal{O}(n), n \in \mathbb{Z}\] of $U$, and such a twist changes the degree of $U$ by $2n$ \cite[p. 9–10]{namba1979families}.
Hence, $ X_{\min}$ is deformation-equivalent to a projective bundle of degree zero if $u$ is even, or to a projective bundle of degree $-1$ if $u$ is odd. When $u=-1$, $ X_{\min}$ is isomorphic to $\CP2$ blown up at a point. We will treat this case in Section~\ref{subsec:XP2}.

Suppose now that $u = 0$. In this case 
\[  X_{\min}\simeq \CP1 \times \CP1\] 
and hence 
\[c_1( X_{\min})=2f+2c.\]  
The intersection form on $X_r$ is
\begin{equation*} 
g_{ij}=g^{ij}=\begin{cases}
1 & \textrm{ if } (i,j) \in \{(0,p),(p,0),(f,c),(c,f)\}, \\
-1 & \textrm{ if } (i,j) \in \{(3,3),\dots,(r+2,r+2)\}, \\
0 & \textrm{ otherwise.}
\end{cases} \end{equation*}

\begin{example}
Consider $X_{\min}\simeq \CP1 \times \CP1$. The fact that \[N_{f+c}= 1\] 
is known classically \cite[Chapter~3, Exercise~3]{kock2007invitation} (see also \cite[Example~7.2]{crauder1995quantum}). For this class,
\[ n_{f+c}=2+2-1=3 .\]
At $0p$ its coefficient is 
\[ 1 - n_{f+c} +2\varepsilon = 1-3+4=2.\]
As multiple covers of the class $f$ or $c$ do not pass through enough general points,
\[N_{kf}=N_{kc}=0, \quad \textrm{when } k>1.\]
The minimal entries of $K$ are hence 
\[
\begin{pmatrix}
0 & 2q_f & 2q_c & 2q_fq_ct_p \\
2& -q_fq_c\frac{t_p^3}{3} & q_ct_p & 2q_c \\
2& q_ft_p & -q_fq_c\frac{t_p^3}{3} & 2q_f \\
 -2 t_p & 2 & 2  & 0
\end{pmatrix}
\]
The Newton pairs are
\[ q^2,q\lambda^2,\lambda^4. \]
\end{example}

\begin{example} 
Consider now $X_1$ which is obtained by blowing up one point on $\CP1 \times \CP1$. 
As $n_{f+c}=3$ we can use Theorem~\ref{thm:gwburules}~\eqref{thm:gwburules_d} to get that \[N_{f+c,1}= 1.\] This class has $n_{f+c,1}= 2$. Similarly, $N_{f,1}$ and $N_{c,1}$ are both equal to 1. We get that the minimal entries of $K$ are
\[
\begin{pmatrix}
0 & 2q_f & 2q_c& 2q_fq_cq_3t_p & 2q_fq_ct_p \\
2 & -q_fq_c\frac{t_p^3}{3}& q_ct_p & q_cq_3 & 2q_c \\
2& q_ft_p & -q_fq_c\frac{t_p^3}{3} & q_fq_3 & 2q_f \\
-1& q_fq_3 & q_cq_3& -q_3^{-1} & -2q_fq_cq_3t_p \\
 -2 t_p & 2 & 2  & 1 & 0
\end{pmatrix}
\]
Let $\omega_1 \colon X_1 \to \CP1 \times \CP1$ be the blow-down map. 
It is known that $X_1$ is isomorphic to the projective plane blown-up at two generic points $P$ and $Q$ \cite[pp.~479-480]{griffiths1978principles}; this gives another blow-down map $\omega_2\colon X_1 \to \CP2$. 
The exceptional curve $E$ of $\omega_1$ then gets identified with the proper transform $\tilde{H}$ of the line $\overline{PQ}$ under $\omega_2$. 
Moreover, for the base $B$ and fiber $F$ in $\CP1 \times \CP1$, we have the relations 
\[
    \omega_1^* F = E_1 + E, \quad \omega_1^* B = E_2 + E. 
\]
On the other hand, we have 
\[
    \omega_2^* H = \tilde{H} +  E_1 +  E_2. 
\] 
It can then be directly checked using these formulae that performing the base change $(f,c,e)\to (h, e_1, e_2)$ in $H^2(X_1)$ brings the above matrix to the one given in Example~\ref{ex:p2ex2}.
\end{example}

Applying repeatedly Theorem~\ref{thm:gwburules}~\eqref{thm:gwburules_d} one can obtain an analogue of of Proposition~\ref{prop:ggeq0} for the blow-ups $X_r$ of $ X_{\min}$. Instead, we will just use the above identification of $X_1$ with the projective plane blown-up at two generic points to reduce the calculation of $K$ to that of the blow-ups of the projective plane. This will be done Section~\ref{subsec:XP2} below.

\subsection{Case 3: $X_{\min} \simeq \CP2$}
\label{subsec:XP2}

The second cohomology of $\CP2$ is generated by the hyperplane class $h$, which has $h^2=1$. The first Chern class is \[c_1( X_{\min})=3h.\] 
When considering this minimal surface, we will write $q_h$ for $q_1$.
 As is known, for $X_r$
\begin{equation*} 
g_{ij}=g^{ij}=\begin{cases}
1 & \textrm{ if } (i,j) \in \{(0,p),(p,0),(h,h)\}, \\
-1 & \textrm{ if } (i,j) \in \{(2,2),\dots,(r+1,r+1)\}, \\
0 & \textrm{ otherwise.}
\end{cases} \end{equation*}

\begin{lemma}
\label{lem:p2min}
The lowest order terms of $K$ on $X_r$, the projective plane blown-up at $r$ points are as follows.
\begin{enumerate}
\item $-q_h\frac{t_p^2}{2}$ at $hh$
\item $q_ht_p$ at $0h$ and $hp$
\item $\pm q_hq_i\left(q_h\frac{t_p^4}{12}+\sum_{\substack{2 \leq j \leq r+1 \\ j \neq i}} q_j\right)$ at $hi$ and $ih$, $2 \leq i \leq r+1$
\item $-q_hq_iq_j$ at $ij$, $2 \leq i,j \leq r+1$
\item $-q_{i}^{-1}$ at $ii$, $2 \leq i \leq r+1$
\item $\pm 2q_hq_i$ at $0i$ and $ip$, $2 \leq i \leq r+1$
\item $3q_h$ at $0p$
\end{enumerate}
\end{lemma}
\begin{proof}
Theorem~\ref{thm:gwburules}, which in this case goes back to \cite{gottsche1998quantum}, gives a straightforward way to calculate the lowest order terms of the operator $K$ for any $X_r$, $r \geq 0$.
One can start from the numbers $N_h=1$ and $n_{h}=3-1=2$.
Then we get that 
\[ N_{h,(1)}=1, \quad n_{h,(1)}=1 \]
and
\[ N_{h,(1,1)}=1, \quad n_{h,(1,1)}=0. \]
But the class $h,(1)$ cannot appear at the entries $ij$, $1\leq i,j \leq m+r$ because its coefficient there would be
\[ 1-n_{h,(1)}+2\varepsilon=1-1+0=0\]
The invariant $N_{2h}=2$, for which $n_{2h}=6-1=5$,
implies that 
\[ N_{2h,(1)}=2, \quad n_{2h,(1)}=4 \]
\end{proof}

\begin{example}
\label{ex:p2ex0}
The minimal entries of $K$ for $X_{\min}$ are
\[
    \begin{pmatrix}
     0 & q_ht_p & 3 q_h \\
     3 & -q_h\frac{t_p^2}{2} & q_ht_p  \\
      -2 t_p & 3 & 0
    \end{pmatrix}
\]
The Newton polygon is bounded by 
\[q, \lambda^3.\]
\end{example}

\begin{example} 
\label{ex:p2ex1}
The minimal entries of $K$ for $X_1$, the projective plane blown-up in a point are
\[
    \begin{pmatrix}
     0 & q_h t_p  + 2q_hq_2  & 2 q_h q_2 & 3 q_h \\
     3 & -q_h\frac{t_p^2}{2} & q_h^2q_2\frac{t_p^4}{12} &  q_h t_p  + 2q_hq_2 \\
     -1 & -q_h^2q_2\frac{t_p^4}{12} & -q_2^{-1} & -2 q_h q_2 \\
      -2 t_p & 3 & 1 & 0
    \end{pmatrix}
\]
At the entries $0h$ and $hp$ we have also shown the terms that become leading upon setting $t_p = 0$. This will be relevant in the comparison with small quantum cohomology; see Remark~\ref{remark:small_qh} below.
\end{example}

\begin{example} 
\label{ex:p2ex2}
The minimal entries of $K$ for $X_2$, the projective plane blown-up in two generic points are
\[
    \begin{pmatrix}
     0 &  q_h t_p  + 2q_h(q_2 + q_3) & 2 q_h q_2 &2 q_h q_3 &  3 q_h \\
     3 & -q_h\frac{t_p^2}{2} +q_hq_2q_3 & q_hq_2\left(q_h\frac{t_p^4}{12}+q_3\right) & q_hq_3\left(q_h\frac{t_p^4}{12}+q_2\right) & q_ht_p  + 2q_h(q_2 + q_3)\\
     -1 & -q_hq_2\left(q_h\frac{t_p^4}{12}+q_3\right) & -q_2^{-1} -q_hq_2q_3& -q_hq_2q_3  & - 2 q_h q_2 \\
     -1 & -q_hq_3\left(q_h\frac{t_p^4}{12}+q_2\right)& -q_hq_2q_3 & -q_3^{-1}  -q_hq_2q_3& - 2 q_h q_3 \\
      -2 t_p & 3 & 1& 1 & 0
    \end{pmatrix}
\]
Again, we have also shown the terms that become leading upon setting $t_p = 0$.
\end{example}

\begin{remark}\label{remark:small_qh} The above operators, when $t_p=0$, become the associated matrix of the linear operator $c_1(T_{X_r})$ on the small quantum cohomology of the base space. 
\begin{enumerate}
\item With this specialization, the eigenvalues of our matrices are in a bijection with critical values of the superpotential mirror to $X_r$ in the cases when $X_r$ is del Pezzo, that is, when $r \leq 8$. These critical values were computed in~\cite[Section~3.3]{auroux2006mirror} for $\CP2$ and in~\cite[Example~2.3]{jerby2017exceptional} for all toric del Pezzo surfaces. For example, in Example~\ref{ex:p2ex1} we have three eigenvalues which are multiples of each other by a primitive third root of unity (this comes from the spectrum of $\mathbb{P}^2$), and one additional eigenvalue close to infinity (the spectrum coming from the center of the blow-up). 
\item  In particular, the operator from Example~\ref{ex:p2ex1} at this specialisation equals
\[
	\begin{pmatrix}
		0 & 2q_hq_2 & 2q_hq_2& 3q_h \\
		3 & 0 & 0 &2q_hq_2 \\
		-1& 0 & -q_2^{-1} & -2q_hq_2 \\
		0 & 3 & 1 & 0
	\end{pmatrix}.
	\]
After a change of basis in the cohomology $H^{\ast}(X_1)$, this coincides with results that appeared earlier \cite[Example~7.3]{crauder1995quantum}.
\item Similarly, by further plugging $q_i=1$ for all $i$, the operator from Example~\ref{ex:p2ex2} specialises to the one in \cite[Example~2.5]{hu2021gamma}.
\item All the remaining terms in the above examples contain $t_p$. So they do not affect the results when
we restrict to the small quantum cohomology by letting $t_p = 0$.
\end{enumerate}
\end{remark}

Similarly as in the ruled surface case, Lemma~\ref{lem:p2min} implies the following.
\begin{lemma} 
\label{lem:gneq1rgeq2}
The degree of the minimal terms of $K$ does not depend on $r$ in the following sense:
\[ \min \deg K_{ij} = \min \deg \overline{K}_{ij}, \quad i,j \in \{0,f,c,3,\dots,r+1,p\}.\]
\end{lemma}

\begin{proof}[Proof of Theorem~\ref{mainthm1} for rational surfaces]
Using Lemmas~\ref{lem:p2min} and~\ref{lem:gneq1rgeq2} the proof is the similar as it was in the ruled surface case. In particular, the inequality \eqref{eq:mindegineqruled} is satisfied in this case as well. The monomial $\lambda^{3+r}$ is always the unique summand of $\chi_{X_r}$ corresponding to the lattice point $(3+r,0)$. Again, the monomials corresponding to the lattice points $(0,-r+1)$, resp. $(3,-r)$ are also unique: they are
\[ 27q_h\prod_{i=2}^{r+1}q_{i}^{-1},\] resp. \[\lambda^3\prod_{i=2}^{r+1}q_{i}^{-1}.\]
These three vertices hence bound the Newton polygon
and this completes the proof of Theorem~\ref{mainthm1}.
\end{proof}

\section{The spectrum of the blow-up}
\label{sec:spectrum}

First, let us recall the following version of Hensel's Lemma for monic polynomials 
with coefficients in the ring of formal power series in one variable.
\begin{theorem}[{Hensel's Lemma~\cite[Lecture 12]{abhyankar1990algebraic}}]\label{thm:Hensel} Let 
\[ F(x,y)=y^n+A_1(x)y^{n-1}+\dots+A_n(x) \in \CC\llbracket x \rrbracket[y] \]
be a monic polynomial of degree $> 0$ in $y$ with coefficients $A_1(x),\dots,A_n(x) \in \CC\llbracket x \rrbracket$. Assume that \[F(0,y)=g(y)h(y)\] where
\[
g(y)=y^r+b_1y^{r-1}+\dots+b_r \in \CC[y],
\]
\[
h(y)=y^s+c_1y^{r-1}+\dots+c_s \in \CC[y]
\]
are monic polynomials of degrees $r >0$ and $s>0$ in $y$ 
such that $\operatorname{g.c.d.}(g(y),h(y))=1$. 
Then there exist unique monic polynomials
\[
G(x,y)=y^r+B_1(x)y^{r-1}+\dots+B_r(x) \in \CC\llbracket x \rrbracket[y],
\]
\[
H(x,y)=y^s+C_1(x)y^{r-1}+\dots+C_s(x) \in \CC\llbracket x \rrbracket[y]
\]
 of degrees $r >0$ and $s>0$ in $y$ 
 such that \[G(0,y)=g(y), \quad  H(0,y)=h(y)\] and \[F(x,y)=G(x,y)H(x,y).\]
\end{theorem}

We will apply this on the characteristic polynomial of the operator $K$ to infer about the quantum spectrum of $X_r$.
As above, let 
$$
b \colon X_r\to X_{\operatorname{min}}
$$
be a birational morphism of regular surfaces over $\mathbb{C}$, with $X_{\operatorname{min}}$ minimal. 
Recall that we fixed a basis $T_1, \ldots , T_m$ of $H^2(X_{\operatorname{min}}, \mathbb{Z})$ and that these correspond to
parameters $q_1,  \ldots , q_m \in \mathbb{C}$. 
The morphism $b$ is obtained by $r$ successive quadratic transformations (blow-ups), 
with exceptional divisors denoted by $E_1, \ldots , E_r$; the variables $q_{m+1}, \ldots , q_{m+r}$ are associated to these divisors. 

\begin{lemma} With the notations of Proposition~\ref{prop:Kij},
\label{lem:conv}
\[ \lim \sum_{\beta,a}k_{\beta,a}q^{\beta}q_e^{a}\frac{t_p^{n_{\beta,a}-\varepsilon}}{(n_{\beta,a}-\varepsilon)!} = 0 \]
as  $|t_p| \ll \infty$ and $(q_1,\dots,q_{m+r})$ converges to $\vec{0}$. 
\end{lemma}
\begin{proof}
We know that the Gromov-Witten potential is analytic in the variables $q,t$ by Lemma~\ref{lem:ruledGWan} and Theorem~\ref{thm:appmain} on a neighbourhood of 0 not containing the hyperplanes defined by the exceptional variables. The series on the LHS is a third derivative of the potential and it only contains nonnegative exponents of all the variables. Hence, it is in fact  analytic on an open neighbourhood of 0. Therefore, it is uniformly continuous and we can exchange the limit with the summation. As each summand converges to $0$, we obtain the claim.
\end{proof}

Let us denote by 
\[
    \lambda_0 (q_1,  \ldots , q_m,t_p) , \ldots , \lambda_{m+1} (q_1,  \ldots , q_m,t_p)
\]
the spectrum of the operator $K$ of $X_{\operatorname{min}}$, and by 
\[
    \mu_0 (q_1,  \ldots , q_{m+r},t_p) , \ldots , \mu_{m+r+1} (q_1,  \ldots , q_{m+r},t_p)
\]
the spectrum of the operator $K$ of $X_r$. 
Let \[U \subset \operatorname{Spec} (\mathbb{C} [q_1, \ldots , q_{m+r}] )\] be 
any conical open subset for the analytic topology whose closure $\overline{U}$ satisfies 
\[
    \overline{U} \cap \{ q_j = 0 \} = \{ \vec{0} \}, \quad 1 \leq j \leq m+r
\]
We will assume that $t_p$ is such that the $q,t$ lies in the domain of convergence of $F_{X_r}$ (see Proposition~\ref{prop:P2domain} for the plane, in all other cases it is trivial).
\begin{theorem}\label{thm:main}
With the above notations, (and up to a suitable relabeling of the spectra), we have 
\[
    \lim \frac{\lambda_j (q_1,  \ldots , q_m,t_p)}{\mu_j (q_1,  \ldots , q_{m+r},t_p)} = 1 
    \quad (0 \leq j \leq m+1)
\]
and 
\[
    \lim \mu_j (q_1,  \ldots , q_{m+r},t_p) q_{j-1} = 1 \quad (m+2 \leq j \leq m+r+1)
\]
as  $|t_p| \ll \infty$ and the point $(q_1, \ldots , q_{m+r})$ converges to $\vec{0}$ in $U$. 
\end{theorem}
If moreover $t_p \to 0$, then one gets the \emph{large radius limit point} of \cite[Section~2.3]{coates2015convergence}.
When $-c_1( X_{\min})$ is nef, the statement holds even without taking limit as it is seen directly from Proposition~\ref{prop:c1antinef}.

By Theorem~\ref{thm:Hensel}, the above means that the characteristic polynomial of $K$ breaks up, in the conical neighbourhood of $\vec{0}$, into two terms: one that can be identified with the characteristic polynomial of the minimal model and one that comes from the blow-up points. This proves Conjecture~\ref{conj:BU} in the case of smooth surfaces.

\begin{proof}
When $X_{\min}$ is a ruled surface or the projective plane, we will prove Theorem~\ref{thm:main} using induction. So assume that that the statement is true for $r-1$. For $i,j \in \{0, 1,\dots,m+r,p\}$, let us denote by $C_{ij}$ the minor of the $ij$ entry of $K-\lambda I$. For ease, we will write again $q_e$ for the variable $q_{m+r}$ corresponding to $E_r$. Let $\overline{K}$ denote the operator of $X_{r-1}$. According to Theorem~\ref{thm:appmain}, the matrix coefficients of $K$ are analytic with respect to all variables, including $t_p$. According to~\cite[Theorem~1.8]{kato2013perturbation}, the eigenvalues of $K$ form (possibly multi-valued) analytic functions of $t_p$ over a domain of the form $|t_p|<\varepsilon$ for some $\varepsilon>0$. So we can assume without loss of generality that $t_p\neq 0$.

Consider the Laplace expansion of the determinant of $K$ with respect to the column of the basis element $E_r$. By Proposition~\ref{prop:Kij} this can be written as
\begin{equation*}
\left(-\lambda-q_e^{-1}-\sum_{\beta,a}k_{\beta,a}q^{\beta}q_e^{a}\frac{t_p^{n_{\beta,a}-\varepsilon}}{(n_{\beta,a}-\varepsilon)!}\right)C_{ee} 
+\sum_{i=0}^{m+r-1}(K-\lambda I)_{ie} C_{ie}
-C_{pe}
\end{equation*}

We first claim that all Newton pairs of 
\[(K-\lambda I)_{ie}C_{ie},\quad 0 \leq i \leq m+r-1\] 
lie strictly above the Newton polygon of $K-\lambda I$. Indeed, the Newton pairs of $C_{ie}$ lie on or above the Newton polygon of $\overline{K}-\lambda I$ due to \eqref{eq:mindegineqruled} just as in the proof of  Theorem~\ref{mainthm1} while nonzero terms of $(K-\lambda I)_{ie}$ are at least of degree 2. Therefore, the Newton pairs coming from $K_{ie} C_{ie}$, $0 \leq i \leq m+r-1$ do not affect the 
asymptotic behavior of the spectrum. As a consequence, we need only consider the first and the last terms of the expansion.

Second, let us consider the polynomial in $\lambda$ corresponding to the diagonal entry $ee$. 
This term comes factored as a product of a linear term in $\lambda$ and a determinant. It follows from Proposition~\ref{prop:Kij} and the rules of determinants that $C_{ee}$ is polynomial in $q_e$ and \[C_{ee}\vert_{q_e=0}=\det(\overline{K}-\lambda I).\]

It follows from Lemma~\ref{lem:conv} that the root of the linear factor can be expressed as 
$$
    \mu_{e} = q_{e}^{-1} + O (1).
$$

Finally, let us turn our attention to the terms coming from $C_{pe}$. A monomial whose corresponding lattice point may possibly lie on or below Newton polygon 
 (up to a nonzero scalar) has the form 
\begin{equation}\label{eq:minimal_monomial}
    f(q,t,\lambda) q_I^{-1} \lambda^{r-1-|I|}
\end{equation}
for some $I \subseteq \{ m+1, \ldots, m+r-1 \}$ where
\[
    q_I^{-1} = \prod_{i\in I} q_i^{-1}\lambda. 
\]
Here $f(q,t,\lambda)$ is a summand of the minor \[|(K-\lambda I)_{ij}|_{i \in \{0,1,\dots,m,m+r\}, j \in \{0,1,\dots,m,p\}}.\]

Spelling out this submatrix one gets for the lowest degree terms
\[
\begin{pmatrix}
     -\lambda & q_ht_p & 3q_h \\
     3 & -\lambda+q_ht_p^2 & q_ht_p \\
     -1 & q_hq_e \left(q_h+\sum_{j=2}^{r} q_j\right) & q_h q_e \\
    \end{pmatrix}
\]
when $X_{\min} =\CP2$
or
\[
\begin{pmatrix}
     -\lambda & 2q_f& 0 & 0 \\
     2-2g-u & -\lambda & 0  & 0 \\
     2 & q_ft_p & -\lambda & 2q_f \\
     -1 & -q_fq_e& 0 & -\frac{u}{2}q_f^2q_e^2
    \end{pmatrix}
\]
when $X_{\min}$ is a ruled surface. A quick computation then gives that the possible values of $f(q,t,\lambda)$, after throwing away those that are surely not minimal and setting $t_p=1$, are
\begin{equation}
\label{eq:Cpelat}
f(q,1,\lambda)=
\begin{cases}
\lambda^2q_hq_e \textrm { or } \lambda q_h \textrm { or } q_h^2 & \textrm { if } X_{\min} =\CP2 \\
\lambda^3 q_f^2q_e^2 \textrm { or } \lambda q_f^3q_e^2 & \textrm { if } X_{\min} \textrm{ is a ruled surface}
\end{cases}
\end{equation}

The case $I=\{ m+1, \ldots, m+r-1 \}$ gives rise to 
\begin{equation}\label{eq:minimal_product}
        f(q,1,\lambda) \prod_{i=m+1}^{m+r-1} q_i^{-1} . 
\end{equation}
The assumption on the open set $U$ means that there exists $C>1$ such that 
$$
    C^{-1} | q_{m+r} | < | q_j | < C | q_{m+r} | \quad \textrm{ for all }\; 1 \leq j \leq m+r-1. 
$$
This implies that the lattice point corresponding to~\eqref{eq:minimal_product} in all cases of \eqref{eq:Cpelat} lies strictly above 
the Newton polygon \[(0,-r+1)- (3,-r)- (3+r,0),\]
resp.
\[ (0,-r+2)- (4,-r)- (4+r,0) \]
of $X_r$ when $X_{\min} =\CP2$, resp. when $X_{\min}$ is a ruled surface.

Using the same reasoning, for every other $I \subseteq \{ m+1, \ldots, m+r-1 \}$ the lattice point 
corresponding to the  monomial~\eqref{eq:minimal_monomial} in all cases of \eqref{eq:Cpelat}
is strictly above the appropriate Newton polygon. 
 
Let us now plug $\lambda' = \lambda q_e^{-1}$, and write 
\[
    \det (K - \lambda I) =  \det (K - q_e \lambda'  I) = q_e^{m+r+1} \det (q_e^{-1} K - \lambda'  I) =  q_e^{m+r+1} \chi' (q_e , \lambda' )
\]
with $\chi'$ monic. 
It is easy to see that the slopes of the Newton polygon of $\chi'$ are equal to those of $\det (K - \lambda I)$ minus $1$. 
In particular, since we have seen that the slopes of $\det (K - \lambda I)$ are at most $1$, we get that the slopes of $\chi'$ are at most $0$. 
This means that plugging $q_e = 0$ in $\chi'$ is now allowed, and we get 
\begin{equation}\label{eq:chi_decomp}
        \chi'( 0, \lambda') = g(\lambda' ) h(\lambda' ) 
\end{equation}
where $g' (\lambda' )$ is monic of degree $m+r$ and 
\[
      \quad h(\lambda' ) = \lambda' - 1. 
\]
We may now apply Theorem~\ref{thm:Hensel} to~\eqref{eq:chi_decomp}, and deduce that there exists a decomposition 
\[
\chi'( q_e, \lambda') = G( q_e, \lambda' ) H( q_e, \lambda' ) 
\]
that specializes to~\eqref{eq:chi_decomp} upon plugging $q_e = 0$. 
In particular, we have 
\[
    H( q_e, \lambda' ) = \lambda' - \mu_{m+r+1} 
\]
for some $ \mu_{m+r+1} =  \mu_{m+r+1} (q_1, \ldots , q_{m+r}, t_p)$ converging to $1$ as $q_1, \ldots , q_{m+r}\to 0$ in $U$. 
This then implies that 
\[
    \det (K - \lambda I) = q_e^{m+r} G( q_e, \lambda' ) \left( \lambda -  q_e^{-1} \mu_{m+r+1} \right) + P  ( q_e, \lambda )
\]
where $P  ( q_e, \lambda )$ comes from the off-diagonal entries of row $e$.
By Theorem~\ref{mainthm1} and the calculations of Section~\ref{sec:newtonpoly}, the slopes (and even the coefficients associated to the boundary points) of $q_e^{m+r} G$ are equal to those of $\det(\overline{K}-\lambda I)$, while the terms coming from $P  ( q_e, \lambda )$ are not dominant.
This finishes the induction argument. 
\end{proof}

\appendix

\section{Convergence of Gromov-Witten potentials of rational surfaces}

\subsection{Convergence of potentials}

Let $X$ be a nonsingular projective variety over $\C$. It is natural to ask if the genus 0 Gromov-Witten potential $F_X$ converges to an analytic function. As a consequence of mirror symmetry, $F_X$ is known to be convergent when $X$ is a complete toric variety, a complete flag variatey \cite{hori2000mirror} or a total space of a toric bundle such that the potential of the base variety converges \cite{koto2022convergence}. Convergence is also known for smooth quadric hypersurfaces \cite{hu2022gamma}. 

In this appendix we consider rational surfaces. As we discussed in the main body of the text, blow-ups of the projective plane are essentially the only cases to consider. But these are not toric if the number of points blown-up is sufficiently high. For an arbitrary rational surface $X$ the specialization of $F_{X}$ at $t_p=1$ was shown to converge on a nontrivial domain in \cite[Section~2.4]{itenberg2004logarithmic}. Using a similar but, as we believe, more elementary method we generalise this to the full potential $F_X$. 

\begin{theorem} 
\label{thm:appmain}
Let $X$ be a rational surface. Then the genus zero Gromov-Witten potential $F_X$ converges to an analytic function.
\end{theorem}

Our main ingredients are an asymptotic formula for the GW invariants of the projective plane stated in \cite{francesco1995quantum} and proved rigorously in \cite{zinger2016some} as well as the relevant case of \eqref{eq:Rk} for the plane obtained first in \cite{gottsche1998quantum}.

For $g > 0$, the genus $g$ GW potential $F^g_X$ might depend on an infinite number of variables. In this case the relevant notion is that of \emph{NF-convergence} which means convergence in an appropriate nuclear Fr\'echet space. See \cite[Definition~7.5]{coates2015convergence} for the precise definition.

\begin{corollary} Let $X$ be a rational surface. Then $F_X^g$ is NF-convergent for all $g \geq 0$.
\end{corollary}
\begin{proof}
The quantum cohomology of rational surfaces is known to be formally semi-simple \cite{bayer2004semisimple}. This means that the associated formal Frobenius manifold is semisimple. 
Combining Theorem~\ref{thm:appmain} with \cite[Theorem~1.1]{coates2015convergence} and the remark above it implies NF-convergence of $F_X^g$.
\end{proof}

\subsection{Notations}

Let $X=\CP2$ be the projective plane. Recall that the second homology is generated by the hyperplane class $H \in H_2(X)$. As in the main body of the text, for an $r$-tuple $\alpha=(a_1,\dots,a_r)$ of integers, $(d,\alpha)$ denotes the homology class \[dH-\sum_{i=1}^ra_iE_i\]
Recall as well that
\[ N_{d,\alpha}= I_{(d,\alpha)}(T_p^{n_{d,\alpha}})\]
where
\[ n_{d,\alpha}=3d-|\alpha|-1\] 
When $\alpha$ is empty, that is, when considering GW invariants of $X$, we will just write $n_{d}$ and $N_{d}$.

The (genus 0) Gromov-Witten potential can be written \cite[Page 8]{gottsche1998quantum} as 
\[F(q,t)=\sum_{(d,\alpha)}N_{d,\alpha}q^{d} q^{\alpha} \frac{t_p^{n_{d,\alpha}}}{n_{d,\alpha}!}\]
where 
\[
q^{d}=q_1^{d}, \quad q^{\alpha}= q_{2}^{a_1} \dots q_{r+1}^{a_r}
\]
and the sum is taken over classes $(d,\alpha) \neq 0$ satisfying $n_{d,\alpha} \geq 0$, $d \geq 0$ and $ \alpha \leq d$.

As a slight generalisation of our notions from Section~\ref{subsec:wdvv}, let the symbol $\vdash (d,\alpha)$ denote the set of pairs $((d_1,\beta), (d_2,\gamma))$ satisfying
\begin{enumerate}
    \item[(i)] $(d_1,\beta), (d_2,\gamma) \neq 0$
    \item[(ii)] $(d_1,\beta)+ (d_2,\gamma)= (d,\alpha)$
    \item[(iii)] $n_{d_1,\beta},\,n_{d_2,\gamma} \geq 0$
    
\end{enumerate}

The components of such pairs will be written as
\[ (d_1,\beta)=(d_1,(b_1,\dots,b_r)),\quad (d_2,\gamma)=(d_2,(c_1,\dots,c_r))\]

\subsection{Proof of Theorem~\ref{thm:appmain}}
\label{sec:gen0conv}
Choosing $i=j=1(=h)$ in relation \eqref{eq:Rk} and exploiting the symmetry of the exceptional divisors one gets the relation that is denoted by $R(i)$ in \cite[Theorem~3.6]{gottsche1998quantum}:
if $n_{d,\alpha} \geq 0$, then for any $1 \leq i \leq r$ 
\begin{equation}
\label{eq:relRi}
\begin{aligned}
d^2a_i N_{d,\alpha} &=  (d^2-(a_i-1)^2))N_{d,\alpha-[i]} \\
&+ \sum_{\vdash (d,\alpha-[i]), d_j >0} N_{d_1,\beta}N_{d_2,\gamma}\left(d_1d_2-\sum_{k=1}^rb_kc_k\right) \left(d_1d_2b_ic_i-d_1^2c_i^2\right)\binom{n_{d,\alpha}}{n_{d_1,\beta}}
\end{aligned}
\end{equation}

\begin{lemma}
\label{lem:CSineq}
Let $(d,\alpha)$ be such that $N_{d,\alpha}\neq 0$. Let $((d_1,\beta),(d_2,\gamma)) \vdash (d,\alpha-[i])$ corresponding to a summand in \eqref{eq:relRi}. Then
\[ \left\vert \sum_{k=1}^rb_kc_k \right\vert \leq d_1d_2\]
\end{lemma}
\begin{proof}
The arithmetic genus of the class $(d,\alpha)$ on $X_r$ is determined by
\[ p_{a}(d,\alpha)=\frac{(d-1)(d-2)}{2}-\sum_{i=1}^r \frac{a_i(a_i-1)}{2} \]
As pointed out in \cite[Section~5.2]{gottsche1998quantum}, $N_{d,\alpha}=0$ if $p_{a}(d,\alpha) < 0$.

As $\sum_i a_i \leq 3d-1$,
we have
\begin{equation}
\label{eq:aiineq}
\sum_i a_i^2 \leq d^2-3d+2+\sum_i a_i \leq d^2-3d+2 +3d -1 = d^2+1
\end{equation}
for any $(d,\alpha)$ with $N_{d,\alpha} \neq 0$. In particular, we can have the equality $\sum_i a_i^2 = d^2+1$ only if $\sum_i a_i = 3d-1$.
Similarly, if $N_{d_1,\beta}\neq 0$ and $N_{d_2,\gamma}\neq 0$, then
\[ \sum_i b_i^2 \leq d_1^2+1 \quad \textrm{and} \quad   \sum_i c_i^2 \leq d_2^2+1\]

Suppose that $\sum_k b_k^2 \leq d_1^2$. Using the Cauchy-Schwartz inequality,
\[ \left\vert \sum_{k=1}^rb_kc_k \right\vert \leq \sqrt{\sum_{k=1}^rb_k^2} \sqrt{\sum_{k=1}^rc_k^2} \leq \sqrt{d_1^2}\sqrt{d_2^2+1} \]
for any non-zero summand in the second line of \eqref{eq:relRi}. As the left side is an integer, we in fact have 
\[ \left\vert \sum_{k=1}^rb_kc_k \right\vert \leq \sqrt{d_1^2}\sqrt{d_2^2}=d_1d_2\]
in these cases. By symmetry, we can draw the same conclusion if $\sum_k c_k^2 \leq d_2^2$.

Suppose that $\sum_k b_k^2=d_1^2+1$ and $\sum_k c^2_k=d_2+1$. Then  apply~\eqref{eq:aiineq} on $(a_k)=(b_k+c_k+\delta_{ik})$ to get that
\[
\left(\sum_k b_k^2+c_k^2+2b_kc_k\right)+2b_i+2c_i+1=\sum_{k} (b_k+c_k+\delta_{ik})^2  \leq (d_1+d_2)^2+1=d_1^2+d_2^2+2d_1d_2+1 \\
\]
It follows that in these cases
\[
\sum_kb_kc_k \leq d_1d_2-b_i-c_i-1 \leq d_1d_2
\]
as well.
\end{proof}

\begin{lemma} 
\label{lem:bound}
For any $(d,\alpha)$ and any $i$ such that $a_i > 0$,
\[ N_{d,\alpha} \leq \frac{1}{a_i} N_{d,\alpha-[i]} \]
\end{lemma}
\begin{corollary}
\label{cor:divfact}
For any $(d,\alpha)$,
\[N_{d,\alpha} \leq \prod_{i=1}^r \frac{1}{a_i!} N_d \]
where we used the convention $0!=1$.
\end{corollary}
\begin{proof}
Follows from Lemma~\ref{lem:bound} using induction and Theorem~\ref{thm:gwburules}~\eqref{thm:gwburules_a} and~\eqref{thm:gwburules_d}.
\end{proof}
\begin{proof}[{Proof of Lemma~\ref{lem:bound}}]

The correspondence \[(d_1,\beta) \longleftrightarrow (d_2,\gamma)\] induces an involution on the summands of the second line of ~\eqref{eq:relRi}. We can hence rewrite this line as 
\[
\begin{multlined}
\sum_{\substack{\vdash (d,\alpha-[i]),\, d_i >0 \\ d_1c_i > d_2b_i}} N_{d_1,\beta}N_{d_2,\gamma}\left(d_1d_2-\sum_{k=1}^rb_kc_k\right) \\ \cdot \left( d_1c_i\left(d_2b_i-d_1c_i\right)\binom{n_{d,\alpha}}{n_{d_1,\beta}}+d_2b_i\left(d_1c_i-d_2b_i\right)\binom{n_{d,\alpha}}{n_{d_2,\gamma}}\right)
\end{multlined}
\]
Here we have used that if $d_1c_i = d_2b_i$, then the corresponding summand vanishes. Note that
\[ n_{d_1,\beta}+n_{d_2,\gamma}= n_{d,\alpha-[i]}-1=n_{d,\alpha}\]
and hence
\[ \binom{n_{d,\alpha}}{n_{d_1,\beta}}= \binom{n_{d,\alpha}}{n_{d,\alpha}-n_{d_1,\beta}} = \binom{n_{d,\alpha}}{n_{d_2,\gamma}} \]
By the condition $d_1c_i > d_2b_i$ we always have that
\[ d_1c_i\left(d_2b_i-d_1c_i\right)+d_2b_i\left(d_1c_i-d_2b_i\right) <0\]

Combining these facts with Lemma~\ref{lem:CSineq} it follows that the second line of~\eqref{eq:relRi} is always non-positive. As a consequence,
\[
 N_{d,\alpha}  \leq   \frac{(d^2-(a_i-1)^2))}{d^2a_i}N_{d,\alpha-[i]} = \left(\frac{1}{a_i}-\frac{1}{a_i} \frac{(a_i-1)^2}{d^2}\right) N_{d,\alpha-[i]} \leq \frac{1}{a_i}N_{d,\alpha-[i]}
\]
where we have used at the last inequality that $a_i \leq d$. 
\end{proof}
The precise claim behind Theorem~\ref{thm:appmain} is the following.
\begin{proposition}
\label{prop:P2domain}
The series
\[
\sum_{(d,\alpha)} N_{d,\alpha} q^{d}q^{\alpha} \frac{t_p^{3d-1-|\alpha|}}{(3d-1-|\alpha|)!}\] converges to an analytic function on a region containing
\[ 
\begin{cases}
\frac{4}{5}|q_1|(|t_p|+\sum_{i=1}^r|q_{i+1}|)^3 < 1 \\
q_{i+1}\neq 0, \quad 1 \leq i\leq r 
\end{cases}
\]
\end{proposition}
\begin{proof}
By Theorem~\ref{thm:gwburules} the potential has a finite number of terms with negative exponents corresponding to the exceptional divisors
\[ F(q,t)=\sum_{i=1}^r q_{i+1}^{-1} + F^+(q,t) \]
For the power series part $F^+(q,t)$ we have
\[
\begin{multlined}
\sum_{(d,\alpha)}\left\vert N_{d,\alpha} q_1^{d}q_2^{a_1}\dots q_{r+1}^{a_r} \frac{t_p^{3d-1-|\alpha|}}{(3d-1-|\alpha|)!}\right\vert \\ \leq \sum_{d} |q|^d \left(\sum_{\alpha} \left\vert \frac{1}{a_1!\dots a_r!} q_2^{a_1}\dots q_{r+1}^{a_r} \frac{t_p^{3d-1-|\alpha|}}{(3d-1-|\alpha|)!} \right\vert N_d \right) \\
= \sum_{d} |q_1|^d \frac{N_d}{(3d-1)!} \left(\sum_{\alpha}  \frac{(3d-1)!}{|\alpha|!(3d-1-|\alpha|)!)}\frac{|\alpha|!}{a_1!\dots a_r!} \left\vert q_2^{a_1}\dots q_{r+1}^{a_r} t_p^{3d-1-|\alpha|}\right\vert   \right) \\ 
=
\sum_{d} |q_1|^d \frac{N_d}{(3d-1)!} \left(\sum_{|\alpha|}  \binom{3d-1}{|\alpha|}\left(\sum_{i=1}^r|q_{i+1}|\right)^{|\alpha|}|t_p|^{3d-1-|\alpha|}   \right) \\ = \sum_{d} |q_1|^d \frac{N_d}{(3d-1)!} \left(|t_p|+\sum_{i=1}^r |q_{i+1}|\right)^{3d-1}
\end{multlined}
\]
where at the first inequality we used Corollary~\ref{cor:divfact}, and at the penultimate and last equalities we used the multinomial theorem.

In \cite[Proposition~3]{francesco1995quantum}, \cite[Corollary~3.2]{zinger2016some} it is shown that
\[ \frac{N_d}{(3d-1)!} \]
grows exponentially with $d$ and in particular
\[ \frac{N_d}{(3d-1)!} \leq \frac{45}{16} \left(\frac{4}{5}\right)^d d^{-\frac{7}{2}}\]
Therefore,
\[
\sum_{(d,\alpha)}\left\vert N_{d,\alpha} q_1^{d}q_2^{a_1}\dots q_{r+1}^{a_r} \frac{t_p^{3d-1-|\alpha|}}{(3d-1-|\alpha|)!}\right\vert \leq \sum_{d} \left( \left(|t_p|+\sum_{i=1}^r |q_{i+1}|\right)^{3d-1}\frac{45}{16} \left(\frac{4|q_1|}{5}\right)^d d^{-\frac{7}{2}} \right)
\]
The latter series converges on the region
\[ \frac{4}{5}|q_1|(|t_p|+\sum_{i=1}^r|q_{i+1}|)^3 < 1 \]
The proposition then follows by Weierstrass' M-test.
\end{proof}

\bibliographystyle{amsplain}
\bibliography{main}

\end{document}